\documentclass[onecolumn]{autart}

\usepackage{amsfonts}
\usepackage{graphicx,amssymb,amstext,amsmath}
\usepackage{graphicx}
\usepackage{epstopdf}
\usepackage{subcaption}
\ifpdf
  \DeclareGraphicsExtensions{.eps,.pdf,.png,.jpg}
\else
  \DeclareGraphicsExtensions{.eps}
\fi

\usepackage{algorithm,algorithmicx,algpseudocode}
\algnewcommand{\algorithmicgoto}{\textbf{go to}}%
\algnewcommand{\Goto}[1]{\algorithmicgoto~\ref{#1}}%
\algnewcommand{\LineComment}[1]{\Statex \(\triangleright\) #1}
\algnewcommand{\LineCommentN}[1]{\Statex \hspace{1cm}\(\triangleright\) #1}

\newenvironment{proof}[1][Proof]{\begin{trivlist}
\item[\hskip \labelsep {\textit{ #1}.}]}{\qed \end{trivlist}}

\newenvironment{remark}[1][Remark.]{\begin{trivlist}
\item[\hskip \labelsep {\bfseries #1}]}{\end{trivlist}}

\newcommand{\TheTitle}{Simultaneous Mode, Input and State Estimation for Switched Linear Stochastic Systems}

\usepackage{amsopn}

\begin{document}

\begin{frontmatter}
\title{\TheTitle}

\author[First]{Sze Zheng Yong}
\author[Second]{Minghui Zhu}
\author[First]{Emilio Frazzoli}

\address[First]{Laboratory for Information and Decision Systems, Massachusetts Institute of Technology, Cambridge, MA 02139, USA (e-mail: szyong@mit.edu, frazzoli@mit.edu).}
\address[Second]{Department of Electrical Engineering, Pennsylvania State University, 201 Old Main, University Park, PA 16802, USA (e-mail: muz16@psu.edu).}

\begin{abstract}
  In this paper, we propose a filtering algorithm for simultaneously estimating the mode, input and state of hidden mode switched linear stochastic systems with unknown inputs. Using a multiple-model approach with a bank of linear input and state filters for each mode, our algorithm relies on the ability to find the most probable model as a mode estimate, which we show is possible with input and state filters by identifying a key property, that a particular residual signal we call \emph{generalized innovation} is a Gaussian white noise. We also provide an asymptotic analysis for the proposed algorithm and provide sufficient conditions for \emph{asymptotically} achieving convergence to the true model (\emph{consistency}), or to the `closest' model according to an information-theoretic measure (\emph{convergence}). 
A simulation example of intention-aware vehicles at an intersection is given to demonstrate the effectiveness of our approach.\end{abstract}

\end{frontmatter}

%

\section{Introduction}
Most autonomous systems must operate without knowledge of the intention and the decisions of other systems or humans. Thus, in many instances, these intentions and control decisions need to be inferred from noisy measurements. This problem can be conveniently considered within the framework of \emph{hidden mode hybrid systems} (HMHS, see, e.g., \cite{verma.delvecchio.12,yong.ACC.2013} and references therein) with unknown inputs, in which the system \emph{state} dynamics is described by a finite collection of functions. Each of these functions corresponds to an intention or \emph{mode} of the hybrid system, where the mode is unknown or \emph{hidden} and mode transitions are autonomous. 
In addition, by allowing unknown inputs in this framework, both deterministic and stochastic disturbance inputs and noise can also be considered. There are a large number of applications, such as urban transportation systems \cite{Yong.Zhu.ea.CDC14_switched}, aircraft tracking and fault detection \cite{Liu.Hwang.2011}, as well attack-resilient estimation of power systems \cite{Yong.Zhu.ea.CDC15}, in which it is not realistic to assume knowledge of the mode and disturbance inputs or they are simply impractical or 
too costly to measure. 

\emph{Literature review.}
The filtering problem of hidden mode hybrid systems without unknown inputs have been extensively studied (see, e.g.,  \cite{Bar-Shalom.2004,Mazor.1998} and references therein), especially in the context of target tracking applications, along with their convergence and consistency properties  \cite{Baram.Jun1978,Baram.Feb1978}. These filtering algorithms, which use a multiple-model approach, consist of a bank of Kalman filters \cite{KalmanF.1960} for each mode and a 
likelihood-based approach that uses the whiteness property of the innovation \cite{Hanlon.2000,Kailath.1968} to determine the probability of each mode. In the case when the mode transition is assumed to be Markovian, hypothesis merging algorithms are developed such as the generalized pseudo-Bayesian (GPBn) as well as the interacting multiple-model (IMM) algorithms \cite{Bar-Shalom.2004,Blom.1988}. 

However, oftentimes the disturbance inputs that include exogenous input, fault or attack signals cannot be modeled as a zero-mean, Gaussian white noise or as a restricted finite set of input profiles, which gives rise to a need for an extension of the existing algorithms to hidden mode hybrid systems with unknown inputs. Such an algorithm was first proposed in \cite{Liu.Hwang.2011} for a limited class of systems, i.e., when unknown inputs only affect the dynamics. Thus, more general algorithms for systems where unknown inputs that can also affect output measurements, as is the case for data injection attacks on sensors \cite{Yong.Zhu.ea.CDC15}, are still lacking.  
Moreover, the approach taken in \cite{Liu.Hwang.2011} is based on running a bank of state-only filters with a possibly suboptimal decoupling of the unknown inputs, as opposed to simultaneous input and state filters that have lately gained more attention. 
Of all the proposed algorithms, the input and state filters in our previous work  \cite{Yong.thesis2015,Yong.Zhu.ea.CDC15_General,Yong.Zhu.ea.Automatica15} are in the most general form and have proven stability and optimality properties, and are hence the most suitable for the problem at hand. 

 \emph{Contributions.}
In this paper, we present a novel multiple-model approach for simultaneous estimation of mode, input and state of switched linear stochastic systems with unknown inputs. As with multiple-model estimation of systems without unknown inputs, a bank of optimal input and state filters \cite{Yong.thesis2015,Yong.Zhu.ea.CDC15_General,Yong.Zhu.ea.Automatica15}, one for each mode, is run in parallel.  Next, we devise a likelihood-based mode association algorithm to determine the probability of each mode. This involves the definition of a \emph{generalized innovation} signal, which we prove is a Gaussian white noise. Then, we use this whiteness property to form a likelihood function, which is used to find the most probable mode. To manage the growing number of hypotheses, we employ a similar approach to the interacting multiple-model estimator \cite{Blom.1988}, which mixes the initial conditions based on mode transition probabilities. 
We then study the asymptotic behavior of our approach (also for a very special case when the hidden mode is deterministic) and provide sufficient conditions for \emph{asymptotically} achieving convergence to the true model (\emph{consistency}), or to the `closest' model according to an information-theoretic measure, i.e., with the minimum Kullback-Leibler (KL) divergence \cite{Kullback.1951}  (\emph{convergence}).  A preliminary version of this 
paper was presented at the 2014 and 2015 IEEE Conference on Decision and Control \cite{Yong.Zhu.ea.CDC14_switched,Yong.Zhu.ea.CDC15} where the asymptotic behavior of only the special case of a deterministic hidden mode was investigated.


\section{Motivating Example} \label{sec:motivation}

\begin{figure}[!b]
\begin{center}
\includegraphics[scale=0.55]{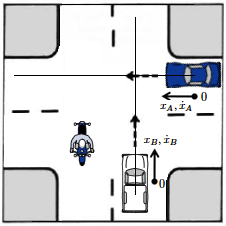}
\caption{Two vehicles crossing an intersection.\label{fig:intersection} }
\end{center}
\end{figure}

To motivate the problem considered in this paper, we consider the scenario of vehicles crossing a 4-way intersection where each vehicle does not have any information about the intention of the other vehicles.
To simplify the problem, we consider the case with two vehicles (see Figure \ref{fig:intersection}): 
Vehicle A is human driven (uncontrolled) and Vehicle B is autonomous (controlled), with dynamics described by $\ddot{x}_A=-0.1\dot{x}_A + d_1$ and $\ddot{x}_B=-0.1\dot{x}_B + u$, where $x$ and $\dot{x}$ are vehicle positions and velocities.  We assume\footnote{The assumed permutation of intentions is for illustrative purposes only and was not a result of any limitations on the proposed algorithms.} that Vehicle A approaches the intersection with a default intention, i.e., without considering the presence of Vehicle B. Then, at the intersection, the driver of Vehicle A can choose between three intentions:
\begin{itemize}
\item to continue while ignoring the other vehicle with an unknown input $d_1$ (\underline{I}nattentive Driver, default mode),
\item to attempt to cause a collision (\underline{M}alicious Driver), or
\item to stop (\underline{C}autious Driver).
\end{itemize}
Then, once either vehicle completes the crossing of the intersection, Vehicle A returns to the default intention. 

Thus, in the presence of noise, this intersection-crossing scenario is an instance of a hidden mode switched linear stochastic system with an unknown input. The intention of driver A is a hidden mode and the actual input of vehicle A is an unknown input (which is not restricted to a finite set). The objective is to simultaneously estimate the intention (mode), input and state of the vehicles for safe navigation through the intersection.

\section{Problem Statement} \label{sec:Problem}
\begin{figure}[!b]
\begin{center}
\includegraphics[scale=0.35]{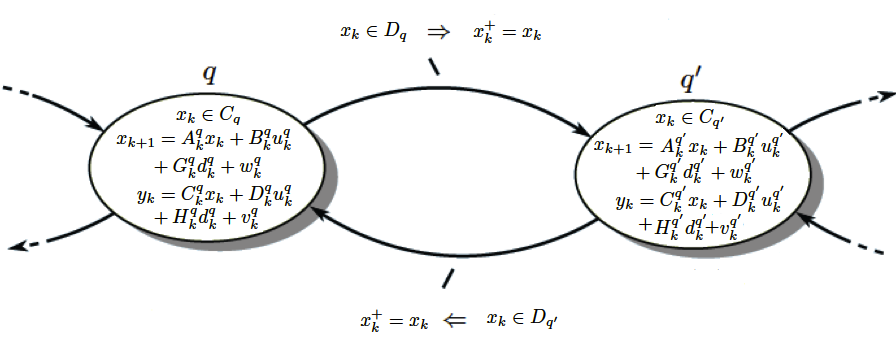}
\caption{Illustration of a switched linear system with unknown inputs as a hybrid automaton with two modes, $q$ and $q'$.\label{fig:hyb_diag} }
\end{center}
\end{figure}

We consider a hidden mode switched linear stochastic system with unknown inputs (see Figure \ref{fig:hyb_diag}):
\begin{align} \label{eq:hybridDyn}
\begin{array}{rl}
(x_{k+1},{q}_k)&=( A_k^{q_k} x_k+B_k^{q_k} u^{q_k}_k+G_{k}^{q_k} d^{q_k}_{k} +w^{q_k}_k,q_k), 
x_k\in \mathcal{C}_{q_k}\\
(x_k,q_k)^+ &=(x_k,\delta^{q_k}(x_k)),  \quad \qquad \quad \qquad \quad \qquad x_k \in \mathcal{D}_{q_k}\\
y_k&=C^{q_k}_{k} x_k+ D^{q_k}_{k} u^{q_k}_k +H^{q_k}_{k} d^{q_k}_k + v^{q_k}_{k} 
\end{array}
\end{align}
where $x_k \in \mathbb{R}^n$ is the continuous system state and $q_k \in \mathcal{Q}\triangleq \{1,2,\hdots,\mathfrak{N}\}$ the hidden discrete state or \emph{mode}. The mode jump process is assumed to be left-continuous and hidden mode systems refer to systems in which $q_k$ is not directly measured and the mode transitions are autonomous. For each mode $q_k$, $u^{q_k}_k \in U_{q_k} \subset \mathbb{R}^m$ is the known input, $d^{q_k}_k \in \mathbb{R}^p$ the unknown input, $y_k \in \mathbb{R}^l$ the output, $\delta^{q_k}(\cdot)$ the mode transition function, $\mathcal{C}_{q_k}$ and $\mathcal{D}_{q_k}$ are flow and jump sets, 
while the process noise $w_k^{q_k} \in \mathbb{R}^n$ and the measurement noise $v^{q_k}_k \in \mathbb{R}^l$ are assumed to be mutually uncorrelated, zero-mean, Gaussian white random signals with known covariance matrices, $Q^{q_k}_k=\mathbb{E} [w_k^{q_k} w_k^{q_k \top}] \succeq 0$ and $R^{q_k}_k=\mathbb{E} [v^{q_k}_k v_k^{q_k \top}] \succ 0$, respectively. The matrices $A^{q_k}_k$, $B^{q_k}_k$, $G^{q_k}_k$, $C^{q_k}_k$, $D^{q_k}_k$ and $H^{q_k}_k$ are known, and $x_0$ is independent of $v^{q_k}_k$ and $w^{q_k}_k$ for all $k$. 
In addition to the common assumptions above, we assume the following:
\begin{description}
\item[A1)]
No prior `useful' knowledge of the dynamics of $d^{q_k}_k$ is known (uncorrelated with $\{d^{q_j}_j\}$, $\forall j\neq k$, and $\{w^{q_j}_j\}, \{v^{q_j}_j\}$, $\forall j$) and $d^{q_k}_k$ can be a signal of any type.
\item[A2)] In each mode, the system is
\emph{strongly detectable}\footnote{That is, the initial condition $x_0$ and the unknown input sequence $\{d^{q_j}_j\} ^{k-1}_{j=0}$ can be asymptotically determined from the output sequence $\{y_i \}^k_{j=0}$ as $k\to \infty$ (see \cite[Section 3.2]{Yong.Zhu.ea.Automatica15} for necessary and sufficient conditions for this property).}. 
\end{description}

The objective of this paper is to design a recursive filter algorithm which simultaneously estimates the system state $x_k$, the unknown input $d^{q_k}_k$ and the hidden mode $q_k$ based on the measurements up to time $k$, $\{y_0,y_1,\hdots, y_k \}$, as well as to analyze the asymptotic behavior of the proposed algorithm.

\section{Preliminary Material} \label{sec:prelim}
In this section, we present a brief summary of the minimum-variance unbiased filter for linear systems with unknown inputs. For detailed proof and derivation of the filter, the reader is referred to \cite{Yong.thesis2015,Yong.Zhu.ea.CDC15_General,Yong.Zhu.ea.Automatica15}. Moreover, we define a \emph{generalized innovation} and show that it is a Gaussian white noise. These form an essential part of the multiple-model estimation algorithm that  we will describe in Section \ref{sec:MainResult}. The algorithm runs a bank of $\mathfrak{N}$ filters (one for each mode) in parallel and the filters are in essence the same except for the different sets of matrices and signals $\{A_k^{q_k},B_k^{q_k},C_k^{q_k},D_k^{q_k},G_k^{q_k},H_k^{q_k},Q_k^{q_k},R_k^{q_k},u_k^{q_k},d_k^{q_k}\}$. Hence, to simplify notation, the conditioning on the mode $q_k$ is omitted in the entire Section \ref{sec:prelim}.

\subsection{Optimal Input and State Filter} \label{sec:ULISE}

As is shown in \cite[Section 3.1]{Yong.Zhu.ea.Automatica15}, the system for each mode after a similarity transformation is given by:
\begin{align}
x_{k+1} & = A_k x_k+B_k u_k+G_{1,k} d_{1,k} +G_{2,k} d_{2,k} +w_k  \label{eq:sysX}\\
z_{1,k}&= C_{1,k} x_k + D_{1,k} u_k +\Sigma_k d_{1,k} + v_{1,k} \label{eq:z1}\\
z_{2,k} &= C_{2,k} x_k + D_{2,k} u_k + v_{2,k}. \label{eq:z2}
\end{align}
The transformation essentially decomposes the unknown input $d_k$ and the measurement $y_k$, each into two components, i.e., $d_{1,k} \in \mathbb{R}^{p_{H_k}}$ and $d_{2,k} \in \mathbb{R}^{p-p_{H_k}}$; as well as $z_{1,k}\in \mathbb{R}^{p_{H_k}}$ and $z_{2,k} \in \mathbb{R}^{l-p_{H_k}}$, where $p_{H_k}=\textrm{rank}(H_k)$.
For conciseness, we assume that the system states can be estimated without delay\footnote{That is, when $C_{2,k} G_{2,k-1}$ has full column rank. By allowing potential delays in state estimation, this assumption can be relaxed such that input and state estimation is possible as long as the system is strongly detectable \cite{Yong.Zhu.ea.CDC15_General}. For brevity, we refer the readers to the filter algorithms and analysis in \cite{Yong.Zhu.ea.CDC15_General}.}. Then, given measurements up to time $k$, the optimal three-step recursive filter in the minimum-variance unbiased sense can be summarized as follows:

\noindent\emph{Unknown Input Estimation}:
\begin{align}
\begin{array}{rl}
\hat{d}_{1,k} &=M_{1,k} (z_{1,k}-C_{1,k} \hat{x}_{k|k}-D_{1,k} u_k)\\
\hat{d}_{2,k-1}&=M_{2,k} (z_{2,k}-C_{2,k} \hat{x}_{k|k-1}-D_{2,k} u_k)\\
\hat{d}_{k-1} &= V_{1,k-1} \hat{d}_{1,k-1} + V_{2,k-1} \hat{d}_{2,k-1} \end{array}
\end{align}
\emph{Time Update}:
\begin{align}
\begin{array}{rl}
 \hat{x}_{k|k-1}&=A_{k-1} \hat{x}_{k-1 | k-1} + B_{k-1} u_{k-1} + G_{1,k-1} \hat{d}_{1,k-1} \\
\hat{x}^\star_{k|k}&=\hat{x}_{k|k-1}+G_{2,k-1} \hat{d}_{2,k-1} 
\end{array}
\end{align}
\emph{Measurement Update}:
\begin{align}
\hat{x}_{k|k}
&= \hat{x}^\star_{k|k} +\tilde{\overline{L}}_k \tilde{\Gamma}_k (z_{2,k}-C_{2,k} \hat{x}^\star_{k|k}-D_{2,k} u_k)  \quad \label{eq:stateEst}
\end{align}
where $\hat{x}_{k-1|k-1}$, $\hat{d}_{1,k-1}$, $\hat{d}_{2,k-1}$ and $\hat{d}_{k-1}$ denote the optimal estimates of $x_{k-1}$, $d_{1,k-1}$, ${d}_{2,k-1}$ and $d_{k-1}$; $\tilde{\Gamma}_k \in \mathbb{R}^{p_{\tilde{R}} \times l-p_{H_k}}$ is a design matrix that is chosen to project the residual signal $\overline{\nu}_k\triangleq z_{2,k}-C_{2,k} \hat{x}^\star_{k|k}-D_{2,k} u_k$ onto a vector of $p_{\tilde{R}}$ independent random variables, while
$\tilde{\overline{L}}_k \in \mathbb{R}^{n \times p_{\tilde{R}}}$, $M_{1,k} \in \mathbb{R}^{p_{H_k} \times p_{H_k}}$ and $M_{2,k} \in \mathbb{R}^{(p-p_{H_k}) \times (l-p_{H_k})}$, as well as $\tilde{L}_k\triangleq \tilde{\overline{L}}_k \tilde{\Gamma}_k$, are filter gain matrices that minimize the state and input error covariances. For the sake of completeness, the optimal input and state filter in \cite{Yong.thesis2015,Yong.Zhu.ea.Automatica15} is reproduced in Algorithm \ref{algorithm1}.

\begin{algorithm}[!b] 
\caption{Opt-Filter ($\hat{x}_{k-1|k-1}^{0,q_k}$,$\hat{d}_{1,k-1}^{0,q_k}$,$P^{x,0,q_k}_{k-1|k-1}$, $P^{d,0,q_k}_{1,k-1}$)\hfill [superscript $q_k$ omitted in the following]}\label{algorithm1}
\begin{algorithmic}[1]
\LineComment{Estimation of $d_{2,k-1}$ and $d_{k-1}$}
\State $\hat{A}_{k-1}=A_{k-1}-G_{1,k-1}M_{1,k-1} C_{1,k-1}$;
\State $\hat{Q}_{k-1}=G_{1,k-1}M_{1,k-1}R_{1,k-1}M_{1,k-1}^\top G_{1,k-1}^\top +Q_{k-1}$;
\State $\tilde{P}_k=\hat{A}_{k-1} P^{x,0}_{k-1|k-1} \hat{A}_{k-1}^\top +\hat{Q}_{k-1}$;
\State $\tilde{R}_{2,k}=C_{2,k} \tilde{P}_k C_{2,k}^\top+R_{2,k}$;
\State $P^d_{2,k-1}=(G_{2,k-1}^\top C_{2,k}^\top \tilde{R}_{2,k}^{-1} C_{2,k} G_{2,k-1})^{-1}$;
\State $M_{2,k}=P^d_{2,k-1} G_{2,k-1}^\top C_{2,k}^\top \tilde{R}^{-1}_{2,k}$;
\State $\hat{x}_{k|k-1}=A_{k-1} \hat{x}^0_{k-1|k-1}+B_{k-1} u_{k-1}+G_{1,k-1} \hat{d}^0_{1,k-1}$;
\State $\hat{d}_{2,k-1}=M_{2,k} (z_{2,k}-C_{2,k} \hat{x}_{k|k-1}-D_{2,k} u_k)$;
\State $\hat{d}_{k-1} =V_{1,k-1} \hat{d}^0_{1,k-1} + V_{2,k-1} \hat{d}_{2,k-1}$;
\State $P^d_{12,k-1}=M_{1,k-1} C_{1,k-1} P^{x,0}_{k-1|k-1} A_{k-1}^\top C_{2,k}^\top M_{2,k}^\top-P^{d,0}_{1,k-1} G_{1,k-1}^\top C_{2,k}^\top M_{2,k}^\top$;
\State $P^d_{k-1}=V_{k-1} \begin{bmatrix} P^{d,0}_{1,k-1} & P^d_{12,k-1} \\ P^{d \top}_{12,k-1} & P^d_{2,k-1} \end{bmatrix} V_{k-1}^\top$;
\LineComment{Time update}
\State $\hat{x}^\star_{k|k}=\hat{x}_{k|k-1}+G_{2,k-1} \hat{d}_{2,k-1}$;
\State $P^{\star x}_{k|k}=G_{2,k-1} M_{2,k} R_{2,k} M_{2,k}^\top G_{2,k}^\top+(I-G_{2,k-1}M_{2,k}C_{2,k})\tilde{P}_k(I-G_{2,k-1}M_{2,k}C_{2,k})^\top$;
\State $\tilde{R}^\star_{2,k}=C_{2,k} P^{\star x}_{k|k} C_{2,k}^\top +R_{2,k} -C_{2,k} G_{2,k-1} M_{2,k} R_{2,k}-R_{2,k} M_{2,k}^\top G_{2,k-1}^\top C_{2,k}^\top$;
\LineComment{Measurement update}
\State $\tilde{L}_k=(P^{\star x}_{k|k} C_{2,k}^\top - G_{2,k-1} M_{2,k} R_{2,k}) \tilde{R}^{\star \dagger}_{2,k}$;
\State $\hat{x}_{k|k}=\hat{x}^\star_{k|k}+\tilde{L}_k(z_{2,k}-C_{2,k} \hat{x}^\star_{k|k}-D_{2,k} u_k)$;
\State $P^x_{k|k}= (I-\tilde{L}_k C_{2,k})G_{2,k-1}M_{2,k} R_{2,k} \tilde{L}_k^\top+ \tilde{L}_k R_{2,k}  M_{2,k}^\top G_{2,k-1}^\top (I-\tilde{L}_k C_{2,k})^\top$
\Statex \hspace{1.5cm} $+(I-\tilde{L}_k C_{2,k}) P^{\star x}_{k|k} (I-\tilde{L}_k C_{2,k})^\top+\tilde{L}_k R_{2,k} \tilde{L}_k^\top$;
\LineComment{Estimation of $d_{1,k}$}
\State $\tilde{R}_{1,k}=C_{1,k}P^x_{k|k}C_{1,k}^\top+R_{1,k}$;
\State $M_{1,k}=\Sigma_k^{-1}$;
\State $P^d_{1,k}=M_{1,k} \tilde{R}_{1,k} M_{1,k}^\top$;
\State $\hat{d}_{1,k}=M_{1,k} (z_{1,k}-C_{1,k} \hat{x}_{k|k}-D_{1,k} u_k)$;
\end{algorithmic}
\end{algorithm}

\subsection{Properties of the Generalized Innovation Sequence}

In Kalman filtering, the innovation reflects the difference between the measured output at time $k$ and the optimal output forecast based on information available prior to time $k$. The \textit{a posteriori} (updated) state estimate is then a linear combination of the \textit{a priori} (predicted) estimate and the weighted innovation. In the same spirit, we generalize this notion of innovation to linear systems with unknown inputs by defining a \emph{generalized innovation} given by:
\begin{align}
\nu_k &\triangleq  \tilde{\Gamma}_k (z_{2,k}-C_{2,k} \hat{x}^\star_{k|k}-D_{2,k} u_k)\triangleq \tilde{\Gamma}_k \overline{\nu}_k \label{eq:g-inno} \\
\nonumber  & =\tilde{\Gamma}_k  (I-C_{2,k} G_{2,k-1} M_{2,k}) (z_{2,k}-C_{2,k} \hat{x}_{k|k-1}-D_{2,k} u_k)
\end{align}
which, similar to the conventional innovation, is weighted by $\tilde{\overline{L}}_k$ and combined with the predicted state estimate $\hat{x}^\star_{k|k}$ to obtain the updated state estimate $\hat{x}_{k|k}$ as seen in \eqref{eq:stateEst}. This definition differs from the conventional innovation in that the generalized innovation uses a subset of the measured outputs, i.e. $z_{2,k}$. In addition, the matrix $\tilde{\Gamma}_k$ is any matrix whose rows are independent of each other and are in the range space of $\mathbb{E}[\overline{\nu}_k \overline{\nu}_k^\top]$ that removes dependent components of $\overline{\nu}_k$ (a consequence of \cite[Lemma 7.6.3]{Yong.thesis2015} and \cite[Lemma 10]{Yong.Zhu.ea.Automatica15}), which further lowers the dimension of the generalized innovation.  An intuition for this is that the information contained in the `unused' subset is already exhausted for estimating the unknown inputs. Moreover, the optimal output forecast that is implied in \eqref{eq:g-inno} is a function of $\hat{x}^\star_{k|k}$ which contains information from the measurement at time $k$. Nonetheless, it is clear from \eqref{eq:g-inno} that when there are no unknown inputs, $z_{2,k}=y_k$, $C_{2,k}=C_k$, $D_{2,k}=D_k$, $G_{2,k-1}=G_{k-1}$ and $\tilde{\Gamma}_k$ can be chosen to be the identity matrix, in which case the definitions of generalized innovation and (conventional) innovation coincide.

In the following theorem, we establish that the generalized innovation, like the conventional innovation, is a Gaussian white noise (see proof in Section \ref{sec:analysis}).

\begin{thm} \label{thm:g-inno}
The generalized innovation, $\nu_k$ given in \eqref{eq:g-inno} is a Gaussian white noise with zero mean and a variance of $S_k=\tilde{\Gamma}_k \tilde{R}^\star_{2,k} \tilde{\Gamma}_k^\top$, with $\tilde{R}^\star_{2,k} \triangleq C_{2,k} P^{\star x}_{k|k} C_{2,k}^\top +R_{2,k} -C_{2,k} G_{2,k-1} M_{2,k} R_{2,k}$.
\end{thm}

\subsection{Likelihood Function}\label{sec:likelihood}

To facilitate the computation of model probabilities that is required in the multiple-model estimation algorithm we propose, 
we derive the \emph{likelihood function} for each mode at time $k$, $q_k$, as follows (proven in Section \ref{sec:analysis}).

\begin{thm} \label{thm:likelihood}
The likelihood  that model $q_k$ is consistent with measurement $z_{2,k}$ and generalized innovation $\nu_k$, given all measurements prior to time $k$, $Z^{k-1}$, is given by the likelihood function:
\begin{align}
\mathcal{L}(q_k|z_{2,k})&\triangleq 
P(z_{2,k}|q_k, Z^{k-1}) = P(\nu^{q_k}_k|Z^{k-1})  =\frac{\exp(-\overline{\nu}_k^{q_k \top} \tilde{R}^{\star,{q_k} \dagger}_{2,k}  \overline{\nu}^{q_k}_k /2)}{(2\pi)^{p_{\tilde{R}}/2} |\tilde{R}^{\star,{q_k}}_{2,k}|_+^{1/2}},
\label{eq:likelihood} 
\end{align}
where $\overline{\nu}_k^{q_k}=(I-C^{q_k}_{2,k} G^{q_k}_{2,k-1} M^{q_k}_{2,k})(z_{2,k}-C^{q_k}_{2,k} \hat{x}^{q_k}_{k|k-1}-D^{q_k}_{2,k} u_k)$, $p_{\tilde{R}}\triangleq \textrm{rank}(\tilde{R}^{\star,{q_k}}_{2,k})$ and $\tilde{R}^{\star,{q_k}}_{2,k}$ is given in Theorem \ref{thm:g-inno}; $\dagger$ and $|\cdot|_+$ represent  the Moore-Penrose pseudoinverse and pseudodeterminant, respectively.
\end{thm}

\section{Multiple-Model Estimation Algorithms} \label{sec:MainResult}

\begin{figure}[!h]
\begin{center}
\includegraphics[scale=0.25]{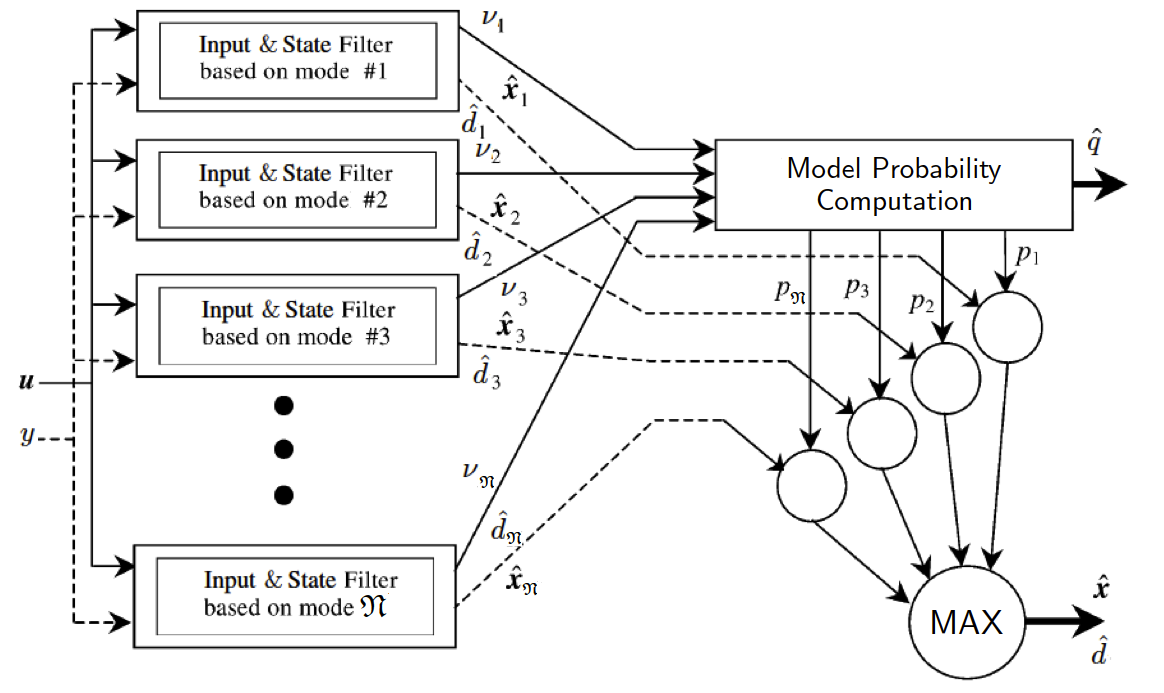}
\caption{Multiple-model framework for hidden mode, input and state estimation.\label{fig8:multipleModel} }
\end{center}
\end{figure}

The multiple-model (MM) approach we take is inspired by the multiple-model filtering algorithms for hidden mode hybrid systems with \emph{known} inputs (e.g.,  \cite{Bar-Shalom.2004,Mazor.1998} and references therein), that have been widely applied for target tracking. Our multiple-model framework consists of the parallel implementation of a \emph{bank of input and state filters} described in Section \ref{sec:ULISE}, with each model corresponding to a system mode (see Figure \ref{fig8:multipleModel}). The objective of the MM approach is then to decide which model/mode is the best representation of the current system mode as well as to  estimate the state and unknown input of the system based on this decision.

To do this, we first use Bayes' rule to recursively find the \emph{posterior} mode probability $\mu^j_k \triangleq  P(q_k=j|Z^k)$ at step $k$ for each mode $j$, given measurements $Z^k=\{z_{1,i},z_{2,i}\}^k_{i=0}$ and \emph{prior} mode probabilities $P(q_k=j'|Z^{k-1}), \, \forall j' \in \{1, \hdots,\mathfrak{N}\}$,  as  
\begin{align}
\begin{array}{rl}
\mu^j_k 
&= P(q_k=j|z_{1,k},z_{2,k},Z^{k-1}) = P(q_k=j|z_{2,k},Z^{k-1})\\
&=\displaystyle\frac{P(z_{2,k}|q_k=j,Z^{k-1}) P(q_k=j|Z^{k-1})}{\sum_{\ell=1}^\mathfrak{N}P(z_{2,k}|q_k=\ell,Z^{k-1}) P(q_k=\ell |Z^{k-1})}, 
\end{array}\label{eq:mu_j}
\end{align}
where we assumed that the probability of $q_k=j$ is independent of the measurement $z_{1,k}$. The rationale is that since we have no knowledge about $d_{1,k}$ and the $d_{1,k}$ signal can be of any type, the measurement $z_{1,k}$ provides no `useful' information about the likelihood of the system mode (cf. \eqref{eq:z1}). The likelihood function is similarly defined as $\mathcal{L}(q_k=j|z_{2,k})\triangleq P(z_{2,k}|q_k=j,Z^{k-1})$ given by \eqref{eq:likelihood}. 
Moreover, the Bayesian approach provides a means to encode what we know about the prior mode probabilities at time $k=0$: 
\begin{align} \label{eq:mu0}
P(q_0=j|Z^0)=\mu^j_0, \quad \forall \ 1,2,\hdots,\mathfrak{N},
\end{align}
where $Z^0$ is the prior information at time $k=0$ and $\sum_{j=1}^\mathfrak{N} \mu^j_0=1$. The maximum \emph{a posteriori} (MAP) mode estimate is then the most probable mode $q_k$ at each time $k$ that maximizes \eqref{eq:mu_j}.

\subsection{Dynamic Multiple-Model Estimation} 

\begin{figure}[!h]
\begin{center}
\includegraphics[scale=0.425]{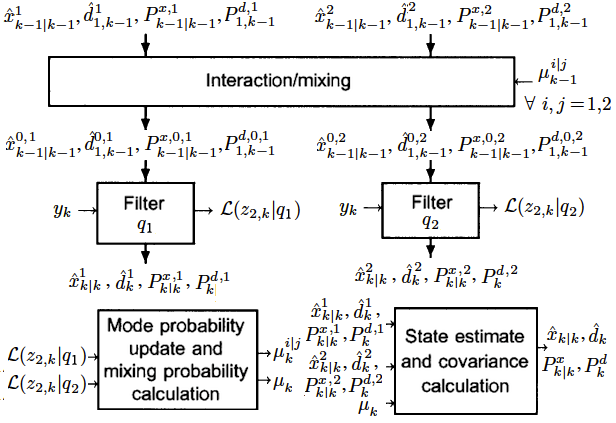}
\caption{Illustration of a dynamic multiple-model estimator with two hidden modes, using two input and state filters as described in Section \ref{sec:ULISE}.\label{fig8:dynamicMM} }
\end{center}
\end{figure}

\begin{algorithm}[!t] 
\caption{Dynamic MM-Estimator ( )}\label{algorithm3}
\begin{algorithmic}[1]
\State Initialize for all $j \in\{1,2,\hdots,\mathfrak{N} \}$: $\hat{x}^j_{0|0}$; $\mu_0^j$;
\Statex $\hat{d}^j_{1,0}=\Sigma_0^{j \, -1} (z^j_{1,0}-C^j_{1,0} \hat{x}^j_{0|0}-D^j_{1,0} u_0)$; 
\Statex $P^{d,j}_{1,0}=\Sigma_0^{j \, -1}(C^j_{1,0} P^{x,j}_{0|0} C^{j \top}_{1,0}+R^j_{1,0})\Sigma_0^{j \, -1}$; 
 \For {$k =1$ to $K$}
  \For {$j =1$ to $\mathfrak{N}$}
   \LineCommentN{Initial Condition Mixing}
   \State $p^j_k=\sum^\mathfrak{N}_{\ell=1} p_{\ell j} \mu_{k-1}^\ell$;
   \For {$i =1$ to $\mathfrak{N}$}
   \State $\mu_k^{i|j}=\frac{p_{ij} \mu_{k-1}^i}{p^j_k};$
   \EndFor
   \State Compute \eqref{eq:xmix}, \eqref{eq:dmix} and \eqref{eq:Pmix};
  \LineCommentN{Mode-Matched Filtering}
  \State Run Opt-Filter($\hat{x}_{k-1|k-1}^{0,j},\hat{d}_{1,k-1}^{0,j},P^{x,0,j}_{k-1|k-1},P^{d,0,j}_{1,k-1}$);
  \State $\overline{\nu}^j_k\triangleq z^j_{2,k}-C^j_{2,k} \hat{x}^{j \star}_{k|k}-D^j_{2,k} u_k;$
  \State $\mathcal{L}(j|z^j_{2,k})=\frac{1}{(2\pi)^{p^j_{\tilde{R}}/2} |\tilde{R}^{j,\star}_{2,k}|_+^{1/2}}\exp \left(-\frac{\overline{\nu}_k^{j \top} \tilde{R}^{j, \star \dagger}_{2,k}  \overline{\nu}^j_k }{2}\right); $
\EndFor
\For {$j =1$ to $\mathfrak{N}$}
 \LineCommentN{Mode Probability Update}
 \State $\mu^j_k=\frac{\mathcal{L}(j|z^j_{2,k}) p_{k}^j}{\sum^\mathfrak{N}_{\ell=1} \mathcal{L}(j|z^\ell_{2,k}) p_{k}^\ell};$
 \LineCommentN{Output}
\State Compute \eqref{eq:outputcombi};
\EndFor
\EndFor
\end{algorithmic}
\end{algorithm}

Our multiple-model estimation algorithm (cf. Figure \ref{fig8:dynamicMM} and Algorithm \ref{algorithm3}) assumes that the hidden mode is stochastic, i.e., the true mode switches in a Markovian manner with known, time-invariant and possibly state dependent transition probabilities
\begin{align*}
P(q_k=j|q_{k-1}=i,x_{k-1})=p_{ij}(x_{k-1}), \ \forall \ i,j \in {1,\hdots,\mathfrak{N}}.
\end{align*}
For brevity and without loss of generality, we assume that the mode transition probabilities are state independent, i.e., $p_{ij}(x_{k-1})=p_{ij}$. In other words, mode transition is a homogeneous Markov chain. 
The incorporation of the state dependency for stochastic guard conditions is rather straightforward, albeit lengthy and interested readers are referred to \cite{Seah.2009} for details and examples. We also assume that we have a fixed number of models. 
For better performance, modifications of the algorithm can be carried out to allow for a varying number of models (cf. \cite{Li.1996} for a discussion on model selection and implementation details).

In fact, the mode transition probabilities can serve as estimator design parameters (cf. \cite{Bar-Shalom.2004}), but care should be given when choosing the mode transition probabilities, as we shall see in Section \ref{sec:dynProp} that a wrong choice can also be detrimental to the consistency of the mode estimates. 
In addition, with the Markovian setting, the mode can change at each time step. As a result, the number of hypotheses (mode history) grows exponentially with time. Therefore, an optimal multiple-model filter is computationally intractable. We thus resort to suboptimal filters that manage the hypotheses in an efficient way. The simplest technique is \emph{hypothesis pruning} in which a finite number of most likely hypotheses are kept, whereas the \emph{hypothesis merging} approach keeps only the last few of the mode histories, and combines hypotheses that differ in earlier steps (cf. \cite{Bar-Shalom.2004} for approaches designed for switched linear systems without unknown inputs). In the following, we propose a hypothesis merging approach similar to the interacting multiple-model (IMM) algorithm \cite{Blom.1988}, 
which is considered the best compromise between complexity and performance \cite{Bar-Shalom.2004}.

Instead of maintaining the exponential number of hypotheses (i.e., $\mathfrak{N}^k$), our estimator maintains a linear number of estimates and filters (i.e., $\mathfrak{N}$) at each time $k$, by introducing three major components: 
\begin{description}
\item[Initial condition mixing:]
We compute the probability that the system was in mode $i$ at time $k-1$ conditioned on $Z^{k-1}$ and currently being in mode $j$:
\begin{align}
\begin{array}{rl}
\mu^{i|j}_k&\triangleq P(q_{k-1}=i|q_{k}=j,Z^{k-1})\\
&=\displaystyle \frac{P(q_k=j|q_{k-1}=i,Z^{k-1})P(q_{k-1}=i|Z^{k-1})}{\sum_{\ell=1}^\mathfrak{N} P(q_k=j|q_{k-1}=\ell,Z^{k-1})P(q_{k-1}=\ell|Z^{k-1})}\\
&=\displaystyle\frac{p_{ij} \mu^i_{k-1}}{P(q_k=j|Z^{k-1})}=\frac{p_{ij} \mu^i_{k-1}}{\sum_{\ell=1}^\mathfrak{N} p_{\ell j} \mu^\ell_{k-1}}.\end{array}\label{eq:mu_ij}
\end{align}
The initial conditions for the filter matched to $q_k=j$ for all $j=\{1,\hdots,N\}$ are then mixed according to:
\begin{align}
&\hspace{0.1cm}\hat{x}^{0,j}_{k-1|k-1}= \textstyle\sum^\mathfrak{N}_{i=1} \mu^{i|j}_k \hat{x}^i_{k-1|k-1} \label{eq:xmix}\\
&\hspace{0.5cm} \hat{d}^{0,j}_{1,k-1}= \textstyle\sum^\mathfrak{N}_{i=1} \mu^{i|j}_k \hat{d}^i_{1,k-1}  \label{eq:dmix}\\
&\hspace{-0.1cm}\begin{array}{l}  \label{eq:Pmix}
P^{x,0,j}_{k-1|k-1} =\sum^\mathfrak{N}_{i=1} \mu^{i|j}_k [(\hat{x}^i_{k-1|k-1}-\hat{x}^{0,j}_{k-1|k-1})(\hat{x}^i_{k-1|k-1}-\hat{x}^{0,j}_{k-1|k-1})^\top\\  \hspace{3.5cm}+P^{x,i}_{k-1|k-1}]\\
\hspace{0.4cm} P^{d,0,j}_{1,k-1}    =\sum^\mathfrak{N}_{i=1} \mu^{i|j}_k [(\hat{d}^i_{1,k-1}-\hat{d}^{0,j}_{1,k-1})(\hat{d}^i_{1,k-1}-\hat{d}^{0,j}_{1,k-1})^\top+P^{d,i}_{1,k-1}]
\end{array}
\end{align}
Note that there is no mixing of $\hat{d}_{2,k}$ and its covariances because they are computed for a previous step and are not initial conditions for the bank of filters.\\ 
\item[Mode-matched filtering:] A bank of $\mathfrak{N}$ simultaneous input and state filters (described in Section \ref{sec:ULISE}) is run in parallel using the mixed initial conditions computed in  \eqref{eq:xmix}, \eqref{eq:dmix} and \eqref{eq:Pmix}. In addition, the likelihood function $\mathcal{L}(q_k=j|z_{2,k})$ corresponding to each filter matched to mode $j$ is obtained using \eqref{eq:likelihood}. \\
\item[Posterior mode probability computation:] Given measurements up to time $k$, the posterior probability of mode $j$  can be found by substituting $P(q_k=j|Z^{k-1})= \sum_{i=1}^\mathfrak{N} p_{i j} \mu^i_{k-1}$ from the denominator of \eqref{eq:mu_ij} into \eqref{eq:mu_j}: 
\begin{align}
\begin{array}{rl}
\mu^{j}_k&=\displaystyle\frac{P(z_{2,k}|q_k=j,Z^{k-1}) \sum_{i=1}^\mathfrak{N} p_{i j} \mu^i_{k-1}}{\sum_{\ell=1}^\mathfrak{N}[P(z_{2,k}|q_k=\ell,Z^{k-1}) \sum_{i=1}^\mathfrak{N} p_{i \ell} \mu^i_{k-1}]}\\
&=\displaystyle\frac{\mathcal{L}(q_k=j|z_{2,k}) \sum_{i=1}^\mathfrak{N} p_{i j} \mu^i_{k-1}}{\sum_{\ell=1}^\mathfrak{N}[\mathcal{L}(q_k=\ell|z_{2,k})\sum_{i=1}^\mathfrak{N} p_{i \ell} \mu^i_{k-1}]}. 
\end{array}\label{eq:mu_dyn}
\end{align} 
Then, these mode probabilities are used to determine the most probable (MAP) mode at each time $k$ and the associated state and input estimates and covariances: 
\begin{align}
\hat{q}_k  &=  \arg \max_{j\in \{1,2,\hdots,\mathfrak{N}\}}  \mu^j_k, \label{eq:outputcombi}\\
\nonumber \hat{x}_{k|k}  &= \hat{x}^{\hat{q}_k}_{k|k},\, 
\  \hat{d}_{k}= \hat{d}^{\hat{q}_k}_{k},\, 
{P}^x_{k|k}= {P}^{x,\hat{q}_k}_{k|k}, \, 
{P}^d_{k}= {P}^{d,\hat{q}_k}_{k}. 
\end{align}
\end{description}

\subsubsection{Filter Properties} \label{sec:dynProp}
We now investigate the asymptotic behavior of our filter, i.e., its \emph{mode distinguishability} properties:

\begin{defn}[Mean Convergence] A filter is \emph{mean convergent} to a model $q \in \mathcal{Q}$, if the geometric mean of the mode probability for model $q$ asymptotically converges to 1 for all initial mode probabilities. 
\end{defn}

\begin{defn}[Mean Consistency] A filter is \emph{mean consistent}, if the geometric mean of the mode probability for the true model $\ast \in \mathcal{Q}$ asymptotically converges to 1 for all initial mode probabilities. 
\end{defn}

In the following, we show that under some reasonable conditions, our filter is \emph{mean convergent} to the model which is closest according to an information-theoretic measure (i.e., with the minimum Kullback-Leibler (KL) divergence \cite{Kullback.1951}), and when the true model is in the set of models, the filter  is \emph{mean consistent}. 
We will also discuss the optimality of resulting input and state estimates. The proofs of these results will be provided in Section \ref{sec:analysis}.
 
\textbf{\emph{Convergence/Consistency of Mode Estimates.}} We first derive the KL divergence of each model from the true model. Then, we analyze the mean behavior (averaged over all possible states) of the mode estimates. 

\begin{lem}\label{lem:KL}
The KL divergence of model $q \in \mathcal{Q}$ from the true model $q=\ast$ is 
\begin{align}
\begin{array}{rl}
D(f^\ast_\ell \| f^q_\ell)&\triangleq \mathbb{E}_{f^\ast_\ell} \left[\ln \frac{f^\ast_\ell}{f^q_\ell}\right]\\
&\;=\frac{1}{2}(p_{\tilde{R}^q_\ell}-p_{\tilde{R}^\ast_\ell}) \ln 2 \pi + \frac{1}{2} \ln |\tilde{R}_{2,\ell}^{q,\star}|_+- \frac{1}{2} \ln |\tilde{R}_{2,\ell}^{\ast,\star}|_+ \\
&\quad +\frac{1}{2} \mathbb{E}_{f^\ast_\ell}[{\rm tr}(\overline{\nu}^q_\ell \overline{\nu}^{q \top}_\ell (\tilde{R}_{2,\ell}^{q,\star})^ \dagger)] -\frac{1}{2} \mathbb{E}_{f^\ast_\ell}[{\rm tr}(\overline{\nu}^\ast_\ell \overline{\nu}^{\ast \top}_\ell (\tilde{R}_{2,\ell}^{\ast,\star})^ \dagger)]\\
&\;=\frac{1}{2}(p_{\tilde{R}^q_\ell}-p_{\tilde{R}^\ast_\ell}) \ln 2 \pi + \frac{1}{2} \ln |\tilde{R}_{2,\ell}^{q,\star}|_+- \frac{1}{2} \ln |\tilde{R}_{2,\ell}^{\ast,\star}|_+ \\
&\quad +\frac{1}{2} {\rm tr}(\tilde{R}_{2,\ell}^{q|\ast,\star} (\tilde{R}_{2,\ell}^{q,\star})^ \dagger) -\frac{1}{2} {\rm tr}(\tilde{R}_{2,\ell}^{\ast,\star} (\tilde{R}_{2,\ell}^{\ast,\star})^ \dagger), 
\end{array}
\label{eq:div}
\end{align}
where $f^j_\ell$ is a shorthand for $P (z_{2,\ell}|q_\ell=j,Z^{\ell-1})=\mathcal{L}(q_\ell=j|z_{2,\ell})$, $\tilde{R}_{2,\ell}^{\star,q|\ast}\triangleq \mathbb{E}_{f^\ast_\ell}[\overline{\nu}^q_\ell \overline{\nu}^{q \top}_\ell]$ and $p^q_{\tilde{R}_\ell}\triangleq \textrm{rank}(\tilde{R}^{\star,q}_{2,\ell})$. Note that the unknown inputs of each model need not have the same dimension; thus $p^q_{\tilde{R}_\ell}$ can be different for all $q \in \{\mathcal{Q} \cup \ast\}$.
\end{lem}

\begin{thm}[Mean Convergence]\label{thm:mean_conv_dyn}
Suppose the following holds (the true model is denoted $\ast)$: 
\begin{description}
\item[Condition (i)\label{cond5}] 
There exist a time step $T \in \mathbb{N}$ and a unique `closest' model $q \in \mathcal{Q}$ such that $$D(f^\ast_\ell \| f^q_\ell)-\ln {\mu^{q,-}_\ell} <D(f^\ast_\ell \| f^{q'}_\ell) -\ln {\mu^{q',-}_\ell},$$ 
 for all $q' \in \mathcal{Q}, q' \neq q$ for all $\ell \geq T$, with 
 $\mu^{\text{--},j}_\ell\triangleq P(q_\ell=j|Z^{\ell-1})=\sum_{i=1}^\mathfrak{N} p_{i j} \mu^i_{\ell-1}$ for $j=q,q'$. 
\end{description} 
Then, the dynamic multiple-model filter is mean convergent 
to this `closest' model $q \in \mathcal{Q}$ in the set of models.
\end{thm}

\begin{thm}[Mean Consistency]\label{thm:mean_cons_dyn}
If \nameref{cond5} holds for the true model $\ast \in \mathcal{Q}$ (in the set of models), then the dynamic multiple-model filter is mean consistent. 
\end{thm}

Note that \nameref{cond5} has  additional `bias' terms for the KL divergences that result from the introduction of mode transition probabilities, $p_{ij}$. Thus, \nameref{cond5} implies that there exists a \emph{unique} model $q \in \mathcal{Q}$ for all $\ell \geq T$ with a `biased' likelihood function that is closest to the `biased' true model and the other `biased' models are strictly less similar to the `biased' true model, measured in terms of their KL divergences. 
Moreover, Theorem \ref{thm:mean_conv_dyn} implies that if the true model $\ast$ is in the set of models $\mathcal{Q}$ but \nameref{cond5} holds for some $q \in \mathcal{Q}$ where $q \neq \ast$, then the dynamic multiple approach is not mean consistent. This serves as an indication that the `bias', $\ln {\mu^{q,-}_\ell}$, that is introduced into \nameref{cond5} by the mode transition probabilities of the dynamic MM algorithm can negatively influence the mode estimates if incorrectly chosen. On the other hand, if chosen wisely, the mode transition probabilities can increase the convergence rate of the mode estimate to the true model.

\textbf{\emph{Optimality of State and Input Estimates.}}
As discussed earlier, the number of hypotheses (mode history) 
grows exponentially with time and hence, an optimal multiple-model filter is computationally intractable. In fact, it can be shown that the assumption of Markovian mode transitions leads to a corresponding graphical model that is cyclic, for which exact inference algorithms are not known. A common inference algorithm that is employed for such graphs is the ``loopy'' belief propagation (sum-product) algorithm whose convergence is still not well understood \cite{weiss.2000}. Similarly, the dynamic MM filter we propose using hypothesis merging techniques to manage the growing number of hypotheses may also lead to suboptimality of the input and state estimates.  
Nonetheless, it does appear to work well in simulation with suitable choices of the mode transition matrix.

\subsection{Special Case: Static Multiple-Model Estimation} \label{sec:sc}
An important special case for the above dynamic MM estimator is when the true system mode is deterministic and fixed within the time scales of interest, i.e., $p_{ii}=1$ and $p_{ij}=0$ $\forall i,j \in \{1,\hdots,\mathfrak{N}\}, i\neq j$. The implication of this is that, the bank of $\mathfrak{N}$ mode-conditioned simultaneous input and state filters (described in Section \ref{sec:ULISE}) is run independently from each other, since $\mu_k^{i|i} \propto p_{ii} \mu^i_{k-1}$ and $\mu_k^{i|j} = 0$ $\forall i,j \in \{1,\hdots,\mathfrak{N}\}, i\neq j$ in \eqref{eq:mu_ij}. 
However, in order to apply the static MM estimator to the switched linear systems, some heuristic modifications of the static MM estimator are necessary. Firstly, to keep all modes `alive' such that they can be activated when appropriate, an artificial lower bound needs to be imposed on the mode probabilities. Moreover, to deal with unacceptable growth of estimate errors of mismatched filters, reinitialization of the filters may be needed, oftentimes with estimates from the most probable mode.

 This special case is especially useful when no knowledge of mode transitions can be assumed, e.g., in adversarial settings of mode attacks \cite{Yong.Zhu.ea.CDC15}.  More importantly, this special case also has nice properties that can be stronger than for the dynamic MM filter in the previous section. These nice properties will be proven in Section \ref{sec:analysis}:
 
\textbf{\emph{Convergence/Consistency of Mode Estimates.}} In addition to the mean behavior of the mode estimates (Theorems \ref{thm:mean_conv} and \ref{thm:mean_cons}), this special case also allows for the characterization of the behavior of the model probability itself (Theorems \ref{thm:convergence} and \ref{thm:consistency}), but it only applies when the log-likelihood sequence $\big\{\ln \frac{f^j_\ell}{f^i_\ell}\big\}_{\ell=1}^k$ is ergodic (a sufficient condition will be provided in Theorem \ref{thm:ergodicity}), i.e., 
\begin{align} 
 \lim_{k \to \infty} \frac{1}{k}\sum_{l=1}^{k} \ln \frac{f^j_\ell}{f^i_\ell} 
&= \mathbb{E}_{f^\ast} \big[\ln \frac{f^j}{f^i}\big]=D(f^\ast \| f^i)-D(f^\ast \| f^j),  
\label{eq:ergodic} 
\end{align}
where we dropped the subscript $k$ to indicate that the distributions are stationary.
 
 \begin{thm}[Mean Convergence (Static)]\label{thm:mean_conv}
Suppose the true model in not the set of models $\mathcal{Q}$, but there exists a unique `closest' model $q \in \mathcal{Q}$ with minimum KL divergence, i.e., the following holds: 
\begin{description}
\item[Condition (ii)\label{cond2}] The true model $\ast$ is not in the set of models, i.e., $\ast \notin \mathcal{Q}$, but there exist a time step $T \in \mathbb{N}$ and a model $q \in \mathcal{Q}$ such that $D(f^\ast_\ell \| f^q_\ell)<D(f^\ast_\ell \| f^{q'}_\ell)$ 
 for all $q' \in \mathcal{Q}, q' \neq q$ for all $\ell \geq T$, where $D(f^\ast_\ell \| f^q_\ell)$ is given in Lemma \ref{lem:KL}. 
\end{description} 
Then, the static multiple-model filter is mean convergent 
to this `closest' model $q \in \mathcal{Q}$ in the set of models. 
\end{thm}

 \begin{thm}[Mean Consistency (Static)]\label{thm:mean_cons}
Suppose the following condition holds:
\begin{description}
\item[Condition (iii) \label{cond1}] The true model $\ast$ is in the set of models, i.e., $\ast \in \mathcal{Q}$ and there exists a time step $T \in \mathbb{N}$ such that $f^\ast_\ell \neq f^q_\ell$, or equivalently, $D(f^\ast_\ell \| f^q_\ell) \neq 0$ 
 for all $q'\in \mathcal{Q}, q \neq \ast$ for all $\ell \geq T$, where $D(f^\ast_\ell \| f^q_\ell)$ is given in Lemma \ref{lem:KL}. 
\end{description}
Then, the static multiple-model filter is mean consistent. 
\end{thm}

\begin{thm}[Convergence (Static)] \label{thm:convergence}
Suppose the sequence $\big\{\ln \frac{f^j_\ell}{f^i_\ell}\big\}_{\ell=1}^k$ is ergodic and the true model is not in the set of models $\mathcal{Q}$, but there exists a unique `closest' model $q \in \mathcal{Q}$ with minimum KL divergence, i.e., the following holds:
\begin{description}
\item[Condition (iv)\label{cond4}] The true model $\ast$ is not in the set of models, i.e., $\ast \notin \mathcal{Q}$, but there exists a unique `closest' model $q \in \mathcal{Q}$ such that $D(f^\ast \| f^q)<D(f^\ast \| f^{q'})$
 for all $q' \in \mathcal{Q}, q' \neq q$ (cf. \cite[Theorem 3.1]{Baram.Feb1978}).
\end{description}
Then, the filter is convergent, i.e., the model probability of this `closest' model  converges almost surely to 1.
\end{thm}

\begin{thm}[Consistency (Static)] \label{thm:consistency}
Suppose the sequence $\big\{\ln \frac{f^j_\ell}{f^i_\ell}\big\}_{\ell=1}^k$ is ergodic and the following holds: 
\begin{description}
\item[Condition (v) \label{cond3}] The true model $\ast$ is in the set of models, i.e., $\ast \in \mathcal{Q}$ and $f^\ast \neq f^q$, or equivalently, $D(f^\ast \| f^q) \neq 0$ 
 for all $q \in \mathcal{Q}, q \neq \ast$ (cf. \cite[Theorem 3.1]{Baram.Jun1978}).
 \end{description} 
 Then, the filter is consistent, i.e., the model probability of the true model converges almost surely to 1.
\end{thm}

\nameref{cond1} and \nameref{cond3} imply that the likelihood functions for all other models $q \neq \ast$ are not identical to the likelihood function for the true model $q=\ast$ for all $\ell \geq T$. In contrast, when the true model is not in the set of models, \nameref{cond2} and \nameref{cond4} imply that there exists a \emph{unique} model $q \in \mathcal{Q}$ for all $\ell \geq T$ with a likelihood function that is `closest' to the true model and the other models are strictly less similar to the true model, measured in terms of their KL divergences. 

\begin{cor}[Monotone Consistency]\label{cor:monotone}
Even if for some $q \in \mathcal{Q}$, $f^q_\ell = f^\ast_\ell$ happens infinitely often 
(i.e., \nameref{cond3} fails to hold), the posterior model mean probabilities will be no worse than their priors for all $\ell \in \mathbb{N}$.
\end{cor}

\textbf{\emph{Optimality of State and Input Estimates.}}
For the discussion on the optimality of the state and input estimates in this special case, we assume that the true model is in the model set, i.e., $\ast \in \mathcal{Q}$. Otherwise, the state and input estimates corresponding to the most probable model may be biased. The following corollary characterizes the optimality of the state and input estimates when using the multiple-model approach with $\ast \in \mathcal{Q}$. 
\begin{cor} \label{cor:est}
If \nameref{cond1} (or \nameref{cond3}) 
holds, then the state and input estimates in \eqref{eq:outputcombi} converge on average (or almost surely)
 to optimal state and input estimates in the minimum variance unbiased sense.  
\end{cor}

\section{Filter Analysis and Proofs} \label{sec:analysis}

We now furnish the proofs for the properties of generalized innovation (Theorems \ref{thm:g-inno} and \ref{thm:likelihood}) and  
the asymptotic analysis of the algorithms presented in Sections \ref{sec:prelim} and \ref{sec:MainResult}.  
We will also provide some verifiable sufficient conditions  for the ergodicity of the sequence $\big\{\ln \frac{f^j_\ell}{f^i_\ell}\big\}_{\ell=1}^k$ in Section \ref{sec:ergodicity}. 
To aid the analysis for the average model probability behavior, we first find the ratio of the geometric means of model probabilities (denoted $\overline{\mu}^q_k$ for mode $q \in \mathcal{Q}$). 
\begin{lem} \label{lem:ratios}
The ratio of the geometric means of model probabilities (with true mode $\ast$) is given by 
\begin{align} \label{eq:geometric_mean_dyn}
\displaystyle\frac{\overline{\mu}^j_k}{\overline{\mu}^i_k} &=  
 \frac{\mu^{j,-}_k}{\mu^{i,-}_k} \exp  \mathbb{E}_{f^\ast_k} \left[\ln \frac{f^j_k}{f^i_k}\right]  \\ 
 \nonumber &=  \exp [ (D(f^\ast_k \| f^i_k)-\ln {\mu^{i,-}_k})-(D(f^\ast_k \| f^j_k)-\ln{\mu^{j,-}_k})].
\end{align}
In the special case in Section \ref{sec:sc}, 
the ratio reduces to
\begin{align}
\displaystyle\frac{\overline{\mu}^j_k}{\overline{\mu}^i_k}  =\frac{\mu^j_0}{\mu^i_0} \exp \sum_{\ell=1}^k \big(D(f^\ast_\ell \| f^i_\ell)-D(f^\ast_\ell \| f^j_\ell)\big), \label{eq:geometric_mean}
\end{align}
where $\frac{\mu_0^j}{\mu_0^i}$ is the ratio of priors. 
Moreover, if the sequence $\big\{\ln \frac{f^j_\ell}{f^i_\ell}\big\}_{\ell=1}^k$ is ergodic  (i.e., \eqref{eq:ergodic} holds), the ratio of model probabilities 
becomes
\begin{align}
 \lim_{k \to \infty} \frac{\mu^j_\ell}{\mu^i_k} 
=\lim_{k \to \infty} \frac{\mu^j_0}{\mu^i_0} \exp ({ k [D(f^\ast \| f^i)-D(f^\ast \| f^j)]}).\label{eq:ratio_mu}
\end{align} 
\end{lem}
\begin{proof}
The expression in \eqref{eq:geometric_mean_dyn} is obtained by taking the geometric mean of  \eqref{eq:mu_dyn} (averaged over all states). Then, for the special case, \eqref{eq:geometric_mean} is obtained since $p_{ii}=1$ and $p_{ij}=0$ $\forall i,j \in \{1,\hdots,\mathfrak{N}\}, i\neq j$. Moreover, \eqref{eq:ratio_mu} can be found from \eqref{eq:mu_dyn} with $p_{ii}=1$ and $p_{ij}=0$ $\forall i \neq j$  by applying 
the law of large numbers in \eqref{eq:ergodic}.
\end{proof}

\subsection{Proof of Theorem \ref{thm:g-inno}}
To prove the whiteness property of the generalized innovation, we substitute \eqref{eq:z2} into \eqref{eq:g-inno} to obtain 
\begin{align}
\nu_k = \tilde{\Gamma}_k (C_{2,k} \tilde{{x}}^\star_{k|k}+v_{2,k}). \label{eq:nu_k_G}
\end{align}
Since $\mathbb{E}[\tilde{{x}}^\star_{k|k}]=0$ and $\mathbb{E}[v_{2,k}]=0$ for all $k$ as is proven in \cite[Lemma 8]{Yong.Zhu.ea.Automatica15}, it follows that the generalized innovation has zero mean, i.e., $\mathbb{E}[\nu_k]=0$, with covariance 
\begin{align*}
\mathbb{E}[\nu_k \nu_j^\top]=\mathbb{E}[ \tilde{\Gamma}_k (C_{2,k} \tilde{x}^\star_{k|k}+v_{2,k}) (C_{2,j} \tilde{x}^\star_{j|j}+v_{2,j})^\top  \tilde{\Gamma}_j^\top].
\end{align*}
We first show that the above covariance is zero when $k\neq j$. Without loss of generality, we assume that $k>j$. From the properties of the filter, we have $\mathbb{E}[v_{2,k}\tilde{x}^{\star \top}_{j|j}]=\mathbb{E}[v_{2,k} v_{2,j}^\top]=0$, thus the covariance reduces to
\begin{align}
\mathbb{E}[\nu_k \nu_j^\top]= \tilde{\Gamma}_k C_{2,k}( \mathbb{E}[ \tilde{x}^\star_{k|k} \tilde{x}^{\star \top}_{j|j}] C_{2,k}^\top+ \mathbb{E}[ \tilde{x}^\star_{k|k} v_{2,j}^\top]) \tilde{\Gamma}_j^\top. \label{eq:nu_cov}
\end{align}
Next, to evaluate $\mathbb{E}[\tilde{x}^\star_{k|k} \tilde{x}^{\star \top}_{j|j}]$ and $\mathbb{E}[ \tilde{x}^\star_{k|k} v_{2,j}^\top]$, we first evaluate the \textit{a priori} estimation error:
\begin{align}
\begin{array}{l}
\tilde{x}^\star_{k+1|k+1} = x_{k+1}-\hat{x}^\star_{k+1|k+1}\\
=\overline{A}_{k} (I-\tilde{\overline{L}}_{k} \tilde{\Gamma}_{k} C_{2,k}) \hat{x}^\star_{k|k}+(I-G_{2,k} M_{2,k+1} C_{2,k+1})w_{k}-G_{2,k} M_{2,k+1} v_{2,k+1}\\
\quad + G_{2,k} M_{2,k+1} C_{2,k+1} G_{1,k} M_{1,k} v_{1,k} -\overline{A}_k \tilde{\overline{L}}_k \tilde{\Gamma}_k v_{2,k}\\
\triangleq  \Phi_k \tilde{x}_{k|k}^\star + v'_k, \label{eq:xhatstar}
\end{array}
\end{align}
where $\Phi_k$ and $v'_k$ are defined above, while $\overline{A}_k\triangleq (I-G_{2,k} M_{2,k+1} C_{2,k+1}) \hat{A}_k$ and $\hat{A}_k\triangleq A_k-G_{1,k} M_{1,k} C_{1,k}$. Using the state transition matrix of the error system 
\begin{align*}
\Phi_{k|j}=\left\{ \begin{array}{ll} \Phi_{k-1}\Phi_{k-1} \hdots \Phi_j = \Phi_{k|j+1} \Phi_j, & k>j  \\ I, & k=j, \end{array}\right.
\end{align*}
the state estimate error is given by
\begin{align} \label{eq:xtilde_star} 
\textstyle \tilde{x}^\star_{k|k}=\Phi_{k|j} \tilde{x}^\star_{j|j} + \sum^{k-1}_{\ell=j} \Phi_{k|\ell+1} v'_\ell. 
\end{align}  
Thus, from \eqref{eq:xhatstar}, we obtain $\mathbb{E}[v'_\ell \tilde{x}^{\star \top}_{j|j}]=0$ and $\mathbb{E}[v'_\ell v_{2,j}^\top]=0$ when $\ell > j$ (i.e., future noise is uncorrelated with the current estimate error and the current noise) while when $\ell=j$, $\mathbb{E}[v_j' \tilde{x}^{\star \top}_{j|j}]=\overline{A}_j \tilde{\overline{L}}_j \tilde{\Gamma}_j R_{2,j} M_{2,j}^\top G_{2,j-1}^\top$ and $\mathbb{E}[v_j' v_{2,j}^\top]=\overline{A}_j \tilde{\overline{L}}_j \tilde{\Gamma}_j R_{2,j}$. With this and from \eqref{eq:nu_cov}, we can evaluate $\mathbb{E}[\tilde{x}^\star_{k|k} \tilde{x}^{\star \top}_{j|j}]$, $\mathbb{E}[\tilde{x}^\star_{k|k} v_{2,j}^\top]$ and $\mathbb{E}[\nu_k \nu_j^\top]$ as follows:
\begin{align}
&\begin{array}{rl} 
 \mathbb{E}[\tilde{x}^\star_{k|k} \tilde{x}^{\star \top}_{j|j}]\hspace{-0.2cm}&= \Phi_{k|j+1} (\Phi_j P^{\star x}_{j|j} + \overline{A}_j \tilde{\overline{L}}_j \tilde{\Gamma}_j R_{2,j} M_{2,j}^\top G_{2,j-1}^\top)\\ 
 \mathbb{E}[\tilde{x}^\star_{k|k} v_{2,j}^{\top}]  \hspace{-0.2cm}&=-\Phi_{k|j+1} (\Phi_j G_{2,j-1} M_{2,j} R_{2,j} + \overline{A}_j \tilde{\overline{L}}_j \tilde{\Gamma}_j R_{2,j})\\
\Rightarrow \mathbb{E}[\nu_k \nu_j^\top] \hspace{-0.2cm}&=\tilde{\Gamma}_k C_{2,k} \Phi_{k|j+1} (\overline{A}_j \tilde{\overline{L}}_j \tilde{\Gamma}_j R_{2,j} M_{2,j}^\top G_{2,j-1}^\top C_{2,j}^\top \\
& \quad + \Phi_j P^{\star x}_{j|j} C_{2,j}^\top - \Phi_j G_{2,j-1} M_{2,j} R_{2,j} - \overline{A}_j \tilde{\overline{L}}_j \tilde{\Gamma}_j R_{2,j}) \tilde{\Gamma}_j^\top\\
&=\tilde{\Gamma}_k C_{2,k} \Phi_{k|j+1} \overline{A}_j (P^{\star x}_{j|j} C_{2,j}^\top-G_{2,j-1} M_{2,j} R_{2,j}-\tilde{\overline{L}}_j \tilde{\Gamma}_j \tilde{R}^\star_{2,j}) \tilde{\Gamma}_j^\top = 0, \hspace{-0.325cm}
\end{array}\label{eq:autocor}
\end{align}
where $\tilde{R}^\star_{2,j}=C_{2,j} P^{\star x}_{j|j} C_{2,j}^\top +R_{2,j} -R_{2,j} M_{2,j}^\top G_{2,j-1}^\top C_{2,j}^\top-C_{2,j} G_{2,j-1}^\top M_{2,j} R_{2,j}$ and for the final equality, we substituted the filter gain from 
\cite[Theorem 7.6.4]{Yong.thesis2015}:
\begin{align*}
\tilde{\overline{L}}_j=(P^{\star x}_{j|j} C_{2,j}^\top-G_{2,j-1} M_{2,j} R_{2,j}) \tilde{\Gamma}_j^\top (\tilde{\Gamma}_j \tilde{R}^\star_{2,j} \tilde{\Gamma}_j^\top)^{-1}.
\end{align*}
Finally, for $j=k$, we can find $S_k\triangleq \mathbb{E}[\nu_k \nu_k^\top]$ as
\begin{align*}
\begin{array}{l}
S_k = \tilde{\Gamma}_k (C_{2,k} P^{\star x}_{k|k}C_{2,k}^\top-C_{2,k} G_{2,k-1} M_{2,k} R_{2,k}\\
\qquad -R_{2,k} M_{2,k}^\top G_{2,k-1}^\top C_{2,k}^\top+R_{2,k}) \tilde{\Gamma}_k^\top 
=\tilde{\Gamma}_k \tilde{R}^\star_{2,k} \tilde{\Gamma}_k^\top.
\end{array}
\end{align*}
Furthermore, from \eqref{eq:nu_k_G} and \eqref{eq:xtilde_star}, since we assumed that $w_k$ and $v_k$ for all $k$ and $x_0$ are Gaussian, the generalized innovation $\nu_k$ is a linear combination of Gaussian random variables and is thus itself Gaussian. Therefore, we have shown that $\nu_k$ is a Gaussian white noise with zero mean and covariance $S_k$. Moreover, $S_k$ is positive definite since $\tilde{\Gamma}_k$ is chosen such that $S_k$ is invertible \cite[Section 7.6.4]{Yong.thesis2015},\cite[Section 5.4]{Yong.Zhu.ea.Automatica15}. \qed

\begin{remark}
The whiteness of the generalized innovation provides an alternative approach to derive the filter gain $\tilde{\overline{L}}_k$ in \cite{Yong.thesis2015,Yong.Zhu.ea.Automatica15} (as can be seen by setting \eqref{eq:autocor} to zero).
\end{remark}

\subsection{ Proof of Theorem \ref{thm:likelihood}}
To form the likelihood function in Theorem \ref{thm:likelihood}, we exploit the whiteness from property of the the generalized innovation $\nu_k =\tilde{\Gamma}_k\overline{\nu}_k$ from Theorem \ref{thm:g-inno}. 
From this property, we know that the conditional probability density function of $\nu_k$ is given by
\begin{align}
\begin{array}{l}
P(\nu_k|Z^{k-1}) 
=\displaystyle\frac{\exp(-\overline{\nu}_k^\top {\tilde{\Gamma}}^\top_k S_k^{-1} {\tilde{\Gamma}}_k \overline{\nu}_k /2)}{(2\pi)^{p_{\tilde{R}}/2} |S_k|^{1/2}}, \label{eq:disintegration}
\end{array}
\end{align}
where we omitted the conditioning on $q_k$ in this proof for conciseness.
Next, note that if $\tilde{\Gamma}_k$ is chosen as a matrix with orthonormal rows, $\tilde{\Gamma}_k S_k^{-1} \tilde{\Gamma}_k=\tilde{\Gamma}^\top_k (\tilde{\Gamma}_k  \tilde{R}^\star_{2,k} \tilde{\Gamma}_k^\top)^{-1} \tilde{\Gamma}_k$ 
 is the generalized inverse and $|{S}_k|$ the pseudo-determinant of $\tilde{R}^\star_{2,k}$ \cite[pp. 527-528]{Rao.73}. From \cite[Lemma 7.6.3]{Yong.thesis2015}, for the case $p_{\tilde{R}}=l-p < l-p_{H_k}$, we also see that \eqref{eq:disintegration} represents the Gaussian distribution of $\overline{\nu}_k \in \mathbb{R}^{l-p_{H_k}}$ whose base measure is restricted to the $p_{\tilde{R}}$-dimensional affine subspace where the Gaussian distribution is supported. On the other hand, when $H_k$ has full rank (i.e., $p=p_{H_k}$ and $p_{\tilde{R}} =l-p = l-p_{H_k}$), the Gaussian distribution is fully supported in $\mathbb{R}^{l-p}$ and no restriction is necessary.
As shown in \cite[Section 7.6.1]{Yong.thesis2015}, there are multiple ways to choose $\tilde{\Gamma}_k$ and the choice in this theorem is one such instance. 
\qed 

\subsection{Proof of Convergence (Theorems  \ref{thm:mean_conv_dyn}, \ref{thm:mean_conv} and \ref{thm:convergence})}

Theorem \ref{thm:mean_conv_dyn} follows directly from \nameref{cond5} and Lemma \ref{lem:ratios}.
For Theorem \ref{thm:mean_conv}, since \nameref{cond2} holds by assumption, then with $j=q'$ and $i=q$, the summand in the exponent of \eqref{eq:geometric_mean} is always strictly negative, which result in the exponential convergence to zero of the ratios of model mean probabilities of all other models ($q'\in \mathcal{Q}, q' \neq q$) to model $q$. The proof of Theorem \ref{thm:convergence} is similar by using \eqref{eq:ratio_mu} and is omitted for conciseness. \qed

\subsection{Proof of Consistency (Theorems \ref{thm:mean_cons_dyn}, \ref{thm:mean_cons}, \ref{thm:consistency} and Corollary \ref{cor:monotone})}

Theorem \ref{thm:mean_cons_dyn} also follows immediately by the application of \nameref{cond5} to Lemma \ref{lem:ratios}.
To prove Theorem \ref{thm:mean_cons}, we note that since $D(f^\ast_\ell \| f^q_\ell) \geq 0$ with equality if and only if $f^\ast_\ell=f^q_\ell$ (\cite[Lemma 3.1]{Kullback.1951}), then applying \nameref{cond1} with $i=\ast \in \mathcal{Q}$ as the true model and $j \in \mathcal{Q}, j \neq \ast$, the summand in the exponent of \eqref{eq:geometric_mean} is always strictly negative, i.e., $D(f^\ast_\ell \| f^\ast_\ell)-D(f^\ast_\ell \| f^j_\ell)=-D(f^\ast_\ell \| f^j_\ell) < 0$ for all $\ell \geq T$ since $f^\ast_\ell \neq f^j_\ell$ by assumption. This means that the ratios of model mean probabilities of all other models ($j\in \mathcal{Q}, j \neq \ast$) to the true model 
converge exponentially to zero, i.e., the mean probability of the true model converges to 1. Theorem \ref{thm:consistency} and Corollary \ref{cor:monotone} can be similarly shown 
and the proof is omitted for brevity. 
\qed

\subsection{Proof of Optimality (Corollary \ref{cor:est})}

For the true model, the filter gains are chosen such that the error covariance is minimized and that the estimates are unbiased (cf. \cite[Section V]{Yong.Zhu.ea.CDC15_General} and \cite[Section 5]{Yong.Zhu.ea.Automatica15} for a detailed derivation and discussion). Hence, the state and input estimates are optimal in the minimum variance unbiased sense. 
If \nameref{cond1} (or \nameref{cond3}) 
holds, by Theorem  \ref{thm:mean_cons} (or Theorem \ref{thm:consistency}), 
the state and input estimates given by \eqref{eq:outputcombi} also converge on average (or almost surely) 
to the state and input estimates of the true model, which are optimal. \qed

\subsection{Sufficient Condition for Ergodicity} \label{sec:ergodicity}

A sufficient condition for the ergodicity of the sequence $\big\{\ln \frac{f^j_\ell}{f^i_\ell}\big\}_{\ell=1}^k$ (for \eqref{eq:ergodic} and Lemma \ref{lem:ratios} to hold) is the stationarity of the matched and mismatched generalized innovation (i.e., when the model is correctly and incorrectly assumed), as is also shown for multiple-model algorithms when inputs are known  
\cite{Baram.Feb1978,Baram.Jun1978}. The existence of a steady-state behavior of closed loop system that is implied by stationarity suggests that the known and unknown inputs should become constant after a finite time. For verifiable sufficient conditions, the eventually constant unknown inputs are assumed to be known after a finite time $T$. In this case, we assume, without loss of generality, that $u_k=0$ and $d_k=0$ for all $k\geq T$.

\begin{thm}[Ergodicity (Static)] \label{thm:ergodicity}
The log-likelihood sequence $\ln \frac{f^j_k}{f^i_k}$ is ergodic if for each model $q \in \mathcal{Q}, q \neq \ast$, the system is strongly detectable and stabilizable, the known and unknown inputs becomes zero after a finite time $T$ and the mismatched system matrix (i.e., the state transition matrix of $[x_{k} \ \hat{x}^q_k]^\top$ for $q\neq \ast$): 
\begin{align*}
A^{\ast,q}\triangleq \begin{bmatrix}   A^\ast & 0 \\ \begin{array}{c}\hat{A}^q (\tilde{L}^q (I-C_{2}^q G_2^q M_2^q)\\+G_2^q M_2^q)T_2^q C^\ast + G_1^q M_1^q T_1^q C^\ast\end{array}  & \begin{array}{c}\hat{A}^q(I-\tilde{L}^q C_2^q)\\(I-G_2^q M_2^q C_2^q) \end{array} \end{bmatrix}, 
\end{align*}\normalsize
is stable\footnote{This implies that the true model $\ast$ is stable and $(\hat{A}^q,C_2)$ is detectable (satisfied by strong detectability (cf. \cite[Corollary 6.4.7]{Yong.thesis2015})). This sufficient but not necessary condition suggests that the state estimates for model $q$ converge to steady-state even when the model is erroneous/mismatched.} 
(i.e., all its eigenvalues are inside the unit circle), where $\tilde{L}^q$ and $M_2^q$ are steady-state matrices of the input and state filter corresponding to model $q$. 
Moreover, we can compute $\tilde{R}_2^{q|\ast,\star}$(cf. \eqref{eq:div},\nameref{cond4}) as: 
\begin{align*}
\tilde{R}_2^{q|\ast,\star}\triangleq \mathbb{E}[\overline{\nu}^q_k\overline{\nu}^{q\,\top}_k]
=(I-C_2^q G_2^q M_2^q)(C^{\ast,q} \Psi^q C^{\ast,q\top}+R_2) (I-C_2^q G_2^q M_2^q)^\top,
\end{align*}
 where $C^{\ast,q} \triangleq \begin{bmatrix} T_2^q C^\ast & -C_2^q\end{bmatrix}$. $\Psi^q=\lim_{k \to \infty} \Psi^q_k$ is the limiting solution of the Lyapunov function 
\begin{align}
\Psi^q_{k+1}=A^{\ast,q} \Psi^q_k A^{\ast,q \, \top} + W^{\ast,q} \breve{Q} W^{\ast,q\,\top},\label{eq:psi} 
\end{align} 
with $W^{\ast,q}=\begin{bmatrix} I & 0 \\ 0 &\tilde{L}^q+G_2^qM_2^q)T_2^q +G_1^q M_1^q T_1^q\end{bmatrix}$ and $\breve{Q}\triangleq \begin{bmatrix} Q & 0 \\ 0 & R\end{bmatrix}$.
\end{thm}
\begin{proof} The claim is proven by showing that the sufficient conditions for ergodicity when there are no unknown inputs in  \cite[Lemma 3.1]{Baram.Feb1978} also hold for the input and state filter in our case, namely that (i) the $A^{\ast,q}$ matrix 
generating simultaneously the true state $x_k$ and the estimate $\hat{x}^q_k$ for $k\geq T$ with zero inputs: 
\begin{align}
\begin{bmatrix} x_{k+1} \\ \hat{x}^q_{k+1}\end{bmatrix}=A^{\ast,q} \begin{bmatrix} x_{k} \\ \hat{x}^q_{k}\end{bmatrix} + W^{\ast,q} \begin{bmatrix}  w_k \\ v_k \end{bmatrix},\label{eq:Psi}
\end{align}
is stable, and (ii) the limit $\Psi^q=\lim_{k \to \infty} \Psi^q_k$ exists and is finite, where $\Psi^q_{k}\triangleq \mathbb{E}\left[\begin{bmatrix}  x_{k} \\ \hat{x}^q_{k}\end{bmatrix} \begin{bmatrix}  x_{k}^\top & \hat{x}^{q\,\top}_{k}\end{bmatrix} \right]$ is generated by \eqref{eq:psi}. As in \cite{Baram.Feb1978}, the former holds by assumption. To prove the latter, we note that the assumption of strong detectability and stabilizability of each model implies that steady-state $\tilde{L}^q$ and $M_2^q$ matrices exist by  \cite[Theorem 6]{Yong.Zhu.ea.Automatica15}. Since $A^{\ast,q}$ and hence, the state dynamics of $\Psi^q_{k}$ in \eqref{eq:Psi} is stable, the limit $\Psi^q$ exists and is finite, which completes the sufficient conditions needed in \cite[Lemma 3.1]{Baram.Feb1978}. It follows that the sequence $\big\{\ln \frac{f^j_\ell}{f^i_\ell}\big\}^k_{\ell=1}$ is ergodic. 
\end{proof}

\section{Simulation Example} \label{sec:examples}
\begin{figure*}[!t]
\begin{center}
\begin{subfigure}[t]{0.475\textwidth}
\centering
\includegraphics[scale=0.575,trim=7.5mm 8mm 9mm 1mm]{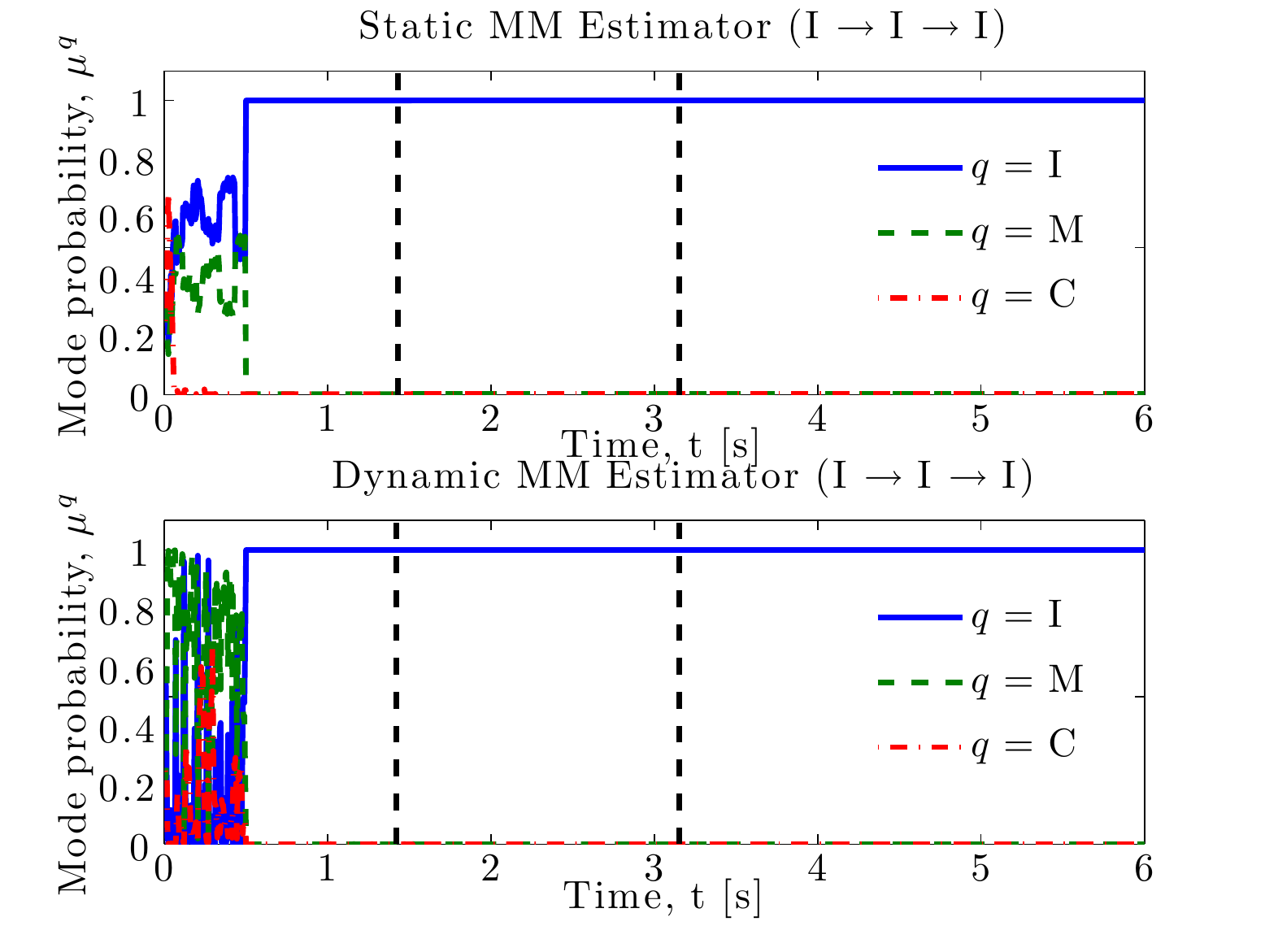}
\caption{Vehicle remain in 'I' mode.} 
\end{subfigure}
\begin{subfigure}[t]{0.475\textwidth}
\includegraphics[scale=0.575,trim=7.5mm 8mm 9mm 1mm]{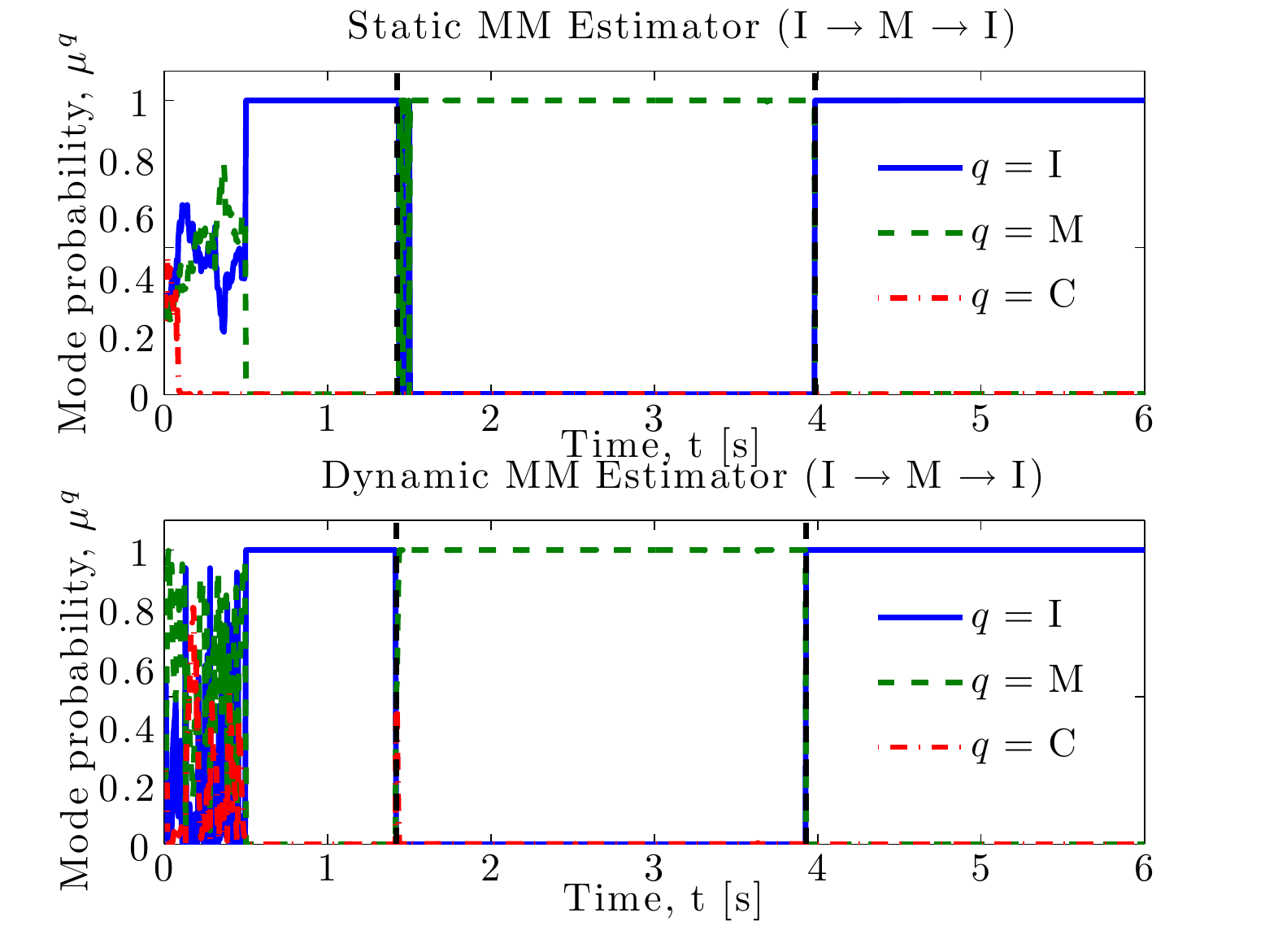}
\caption{Vehicle switches intentions 'I$\rightarrow$M$\rightarrow$I'.}
\end{subfigure}\\
\begin{subfigure}[t]{0.45\textwidth}
\includegraphics[scale=0.575,trim=7.5mm 8mm 9mm -5mm]{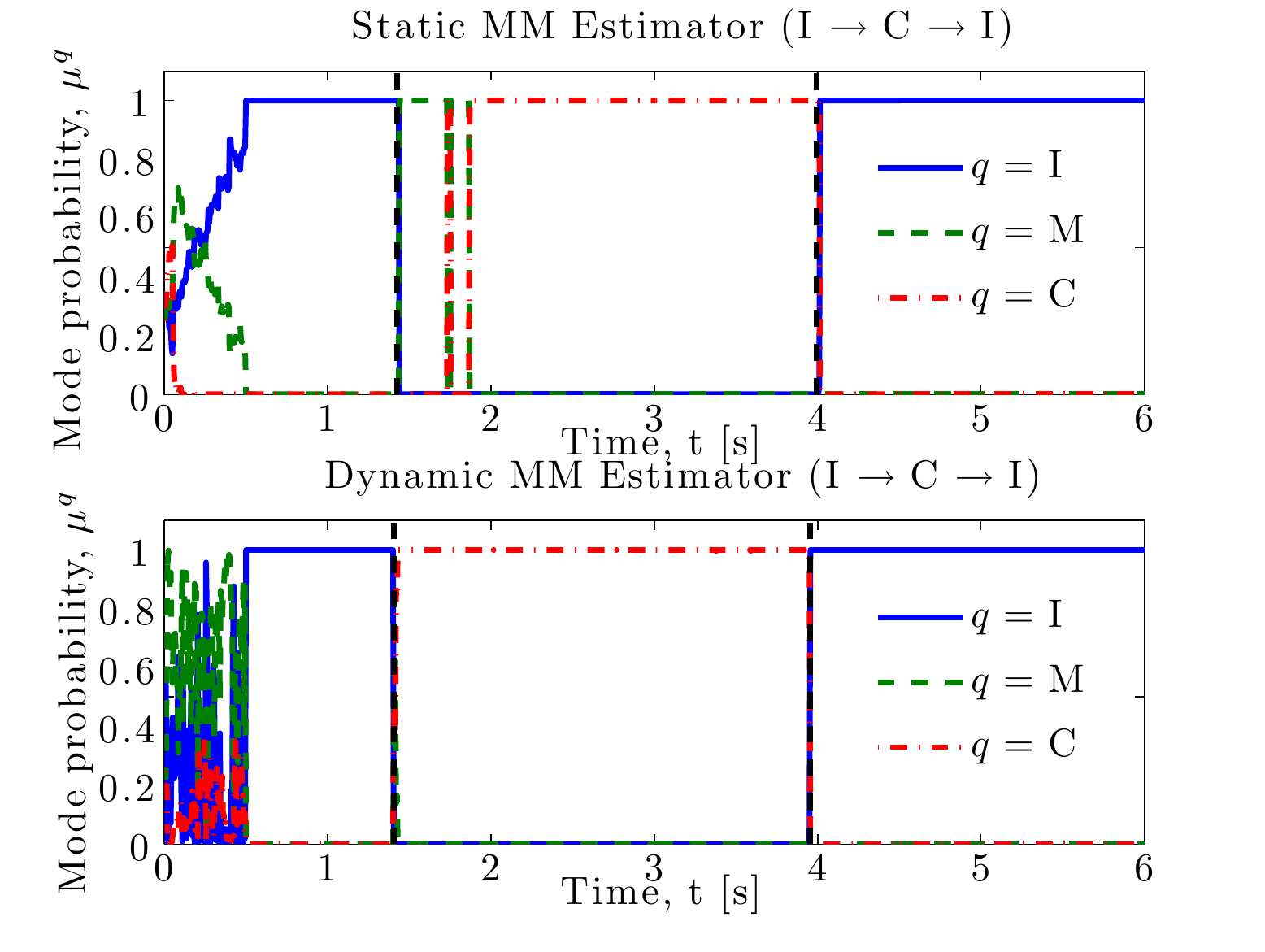}
\caption{Vehicle switches intentions 'I$\rightarrow$C$\rightarrow$I'.}
\end{subfigure}
\caption{Mode probabilities for each mode with static (top) and dynamic (bottom) MM estimators.\label{fig:modes} }
\end{center}
\end{figure*}
We return to the motivating example in Section \ref{sec:motivation} of two vehicles crossing an intersection.
Using the hidden mode system model with state $x=\begin{bmatrix} x_A, \dot{x}_A, x_B, \dot{x}_B \end{bmatrix}$, each intention corresponds to a mode $q \in \{ \textrm{I , M, C} \}$ with the following set of parameters and inputs:

\noindent $\bullet$ Inattentive Driver ($q=\textrm{I}$), with an unknown time-varying $d_1$ (uncorrelated with $x_B$ and $\dot{x}_B$, otherwise unrestricted):
\begin{align*}
\begin{array}{rl}
A^{I}_c&= \begin{bmatrix}  0 & 1 & 0 & 0 \\ 0 & -0.1 & 0 & 0 \\ 0 & 0 & 0 & 1 \\ 0 & 0 & 0 & -0.1 \end{bmatrix}, B_c^{I}=\begin{bmatrix}  0 \\ 0 \\ 0 \\ 1 \end{bmatrix}, G_c^I =\begin{bmatrix}  0 & 0 \\ 1 & 0 \\ 0 & 0 \\ 0 & 0 \end{bmatrix},\\
C^{I}_c&=\begin{bmatrix}  1 & 0 & 0 & 0 \\ 0 & 1 & 0 & -1 \\ 0 & 0 & 1 & 0 \\ 0 & 0 & 0 & 1  \end{bmatrix}, D_c^{I}=\begin{bmatrix}  0 \\ 0 \\ 0 \\ 0 \end{bmatrix}, H_c^I =\begin{bmatrix}  0 & 0 \\ 0 & 0 \\ 0 & 0.1\\ 0 & 1 \end{bmatrix}.\end{array}
\end{align*}
\noindent $\bullet$ Malicious Driver ($q=\textrm{M}$), i.e., with $d_1=K_p (x_B-x_A)+K_d (\dot{x}_B-\dot{x}_A)$ where $K_p=2$ and $K_d=4$:
\begin{align*}
\begin{array}{rl}
A^{M}_c&= \begin{bmatrix}  0 & 1 & 0 & 0 \\ -K_p & -0.1-K_d & K_p & K_d \\ 0 & 0 & 0 & 1 \\ 0 & 0 & 0 & -0.1 \end{bmatrix}, H_c^I =\begin{bmatrix}  0 & 0 \\ 0 & 0 \\ 0 & 0\\ 0 & -1 \end{bmatrix}, \\
B_c^{M}&=B_c^{I}, G_c^M =G_c^I , C^{M}_c=C^{I}_c, D_c^{M}=D_c^{I}.\end{array}
\end{align*}
\noindent $\bullet$ Cautious Driver ($q=\textrm{C}$), i.e., with $d_1=-K_p x_A -K_d \dot{x}_A$ where $K_p=2$ and $K_d=4$:
\begin{align*}
\begin{array}{rl}
A^{M}_c&= \begin{bmatrix}  0 & 1 & 0 & 0 \\ -K_p & -0.1-K_d & 0 & 0 \\ 0 & 0 & 0 & 1 \\ 0 & 0 & 0 & -0.1 \end{bmatrix}, H_c^I =\begin{bmatrix}  0 & 0 \\ 0 & -1 \\ 0 & 0 \\ 0 & 1 \end{bmatrix}, \\
B_c^{M}&=B_c^{I}, G_c^M =G_c^I , C^{M}_c=C^{I}_c, D_c^{M}=D_c^{I}.\end{array}
\end{align*}
Furthermore, the velocity measurement of the vehicle is corrupted by an unknown time-varying bias $d_2$. Thus, the switched linear system is described by
\begin{align*}
\dot{x}=A^q_c x + B^q_c u + G^q_c d +w^q, \quad y=C^{q}_c x + D^q_c u +H^q_c d +v^{q},
\end{align*}
where $d=[ d_1 \ d_2 ]^\top$, the intensities of the zero mean, white Gaussian noises, $w=[0 \ w_1 \ 0 \ w_2]^\top$ and $v$, are
\begin{align*}
Q_c = 10^{-4} \begin{bmatrix}  0 & 0 & 0 & 0  \\ 0 & 1.6 &  0 & 0\\ 0& 0 & 0 & 0 \\ 0 & 0 & 0 & 0.9 \end{bmatrix}; 
R_c = 10^{-4} \begin{bmatrix}  1 & 0 & 0 & 0  \\ 0 & 0.16 &  0 & 0\\ 0& 0 & 0.9 & 0 \\ 0 & 0 & 0 & 2.5 \end{bmatrix}.
\end{align*}
Since the proposed filter is for discrete-time systems, we employ a common conversion algorithm to convert the continuous dynamics to a discrete equivalent model with sample time $\triangle t= 0.01s$, assuming zero-order hold for the known and unknown inputs, $u$ and $d$. 

\begin{figure}[!tp]
\begin{center}
\begin{subfigure}[t]{0.85\textwidth}
\centering
\includegraphics[scale=0.64,trim=12mm  6mm 10mm 7.5mm,clip]{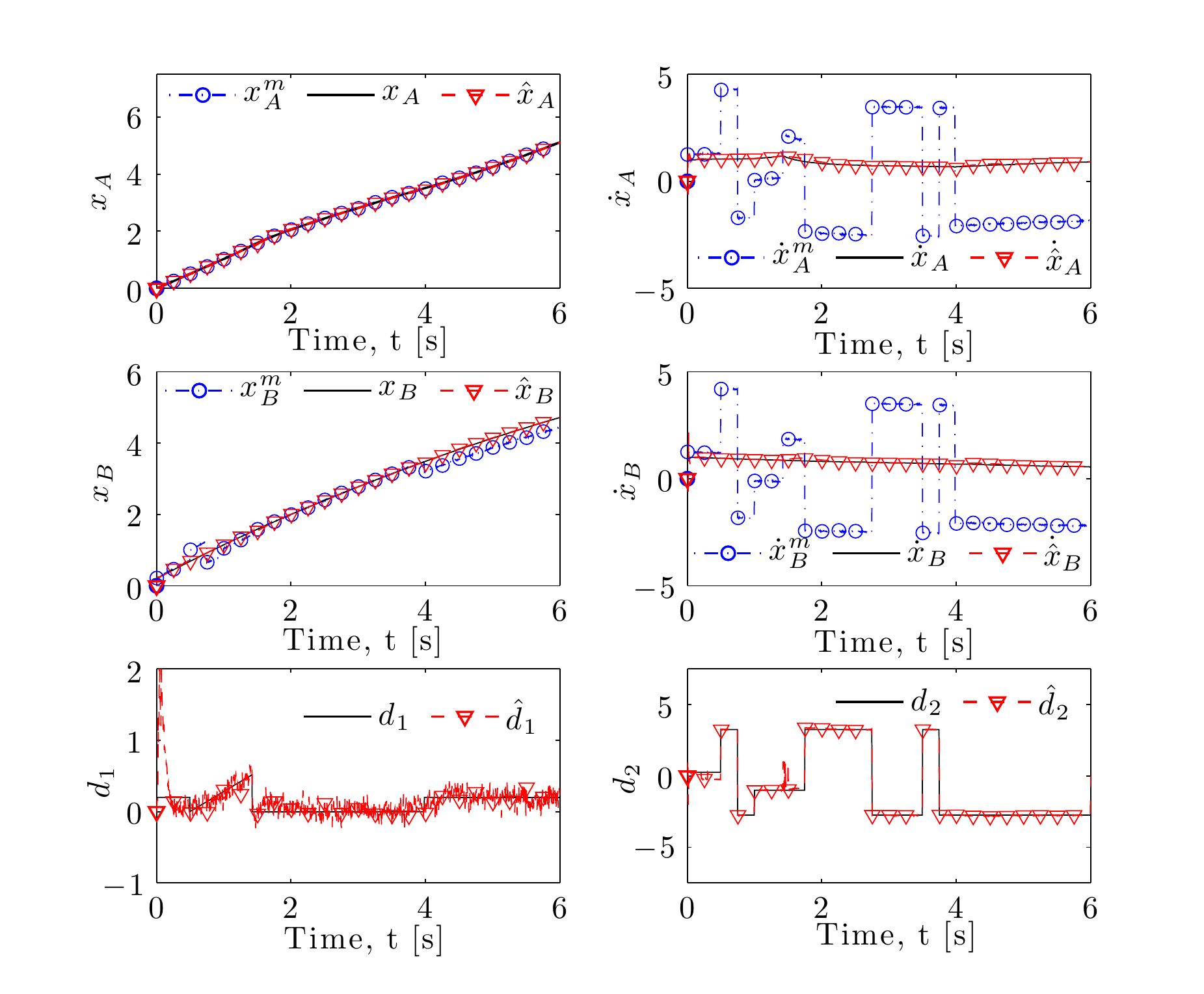} 
\caption{With the \emph{static} MM estimator.\label{fig:static}}
\end{subfigure}\\
\begin{subfigure}[t]{0.85\textwidth}
\centering
\includegraphics[scale=0.64,trim=12mm 5mm 10mm 8mm,clip]{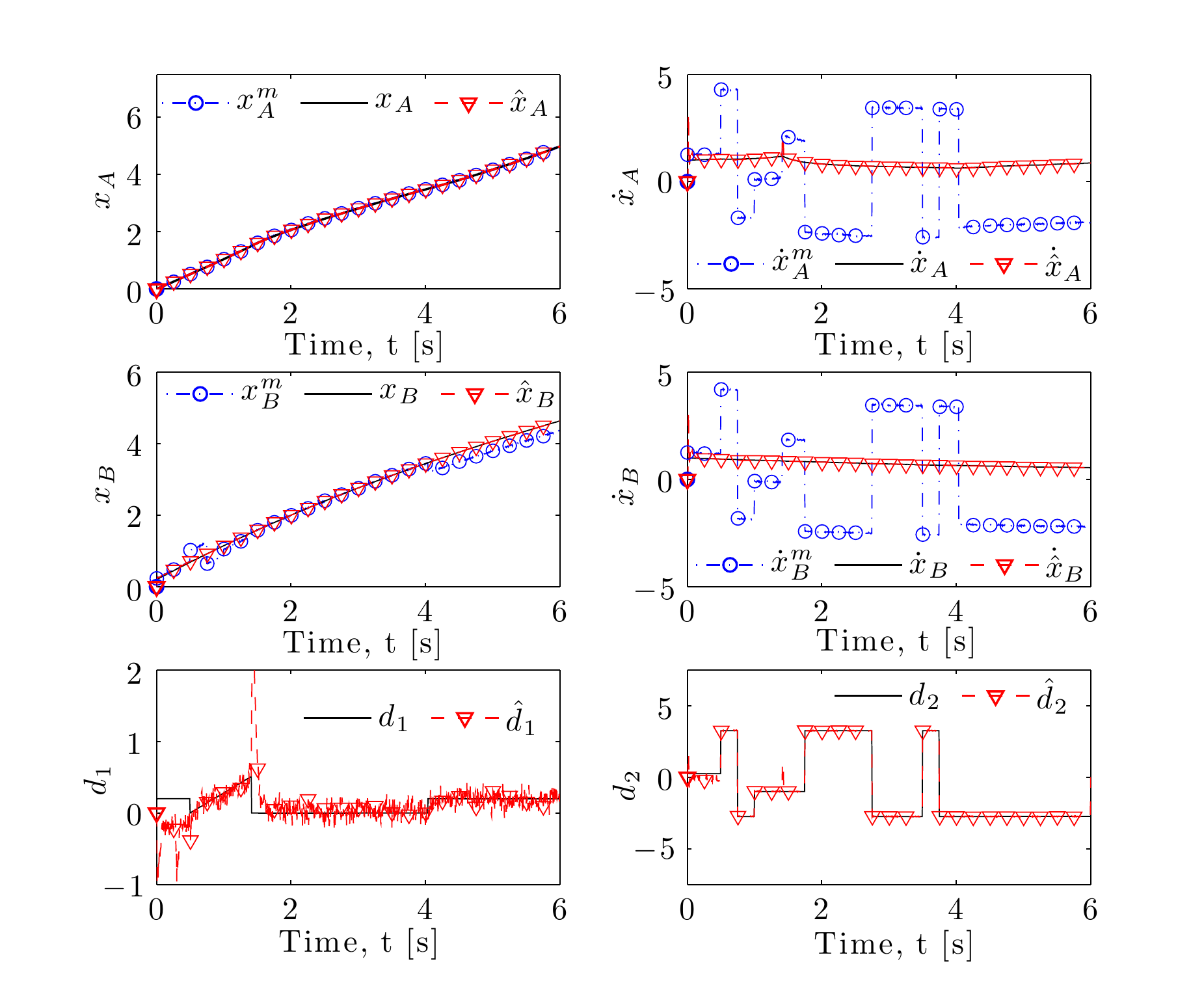} 
\caption{With the \emph{dynamic} MM estimator. \label{fig:dynamic}}
\end{subfigure}
\caption{Measured (superscript `m', unfiltered), actual and estimated states and unknown inputs  for  the 'I$\rightarrow$M$\rightarrow$I' case.}
\end{center}
\end{figure}

From Figure \ref{fig:modes}, we observe that both the static (i.e., the special case in Section \ref{sec:sc}) and dynamic MM estimators were successful at inferring the hidden modes of the system in the cases when the vehicle remains in the `Inattentive' mode, or switches modes according to I$\rightarrow$M$\rightarrow$I or I$\rightarrow$C$\rightarrow$I. The performance of the static MM estimator is slightly worse than the dynamic variant, as can be seen in Figure \ref{fig:modes}(c). On the other hand, the changes in the mode probability estimate of the dynamic MM estimator are quicker which could be interpreted as having a higher `sensitivity' to mode changes.

Taking a closer look at the `I$\rightarrow$M$\rightarrow$I' scenario (the others are omitted due to space  limitations) depicted in Figures \ref{fig:static} and \ref{fig:dynamic}, we observe that both variants of the MM estimators performed satisfactorily in the estimation of states and unknown inputs. Similar to the observation of the mode probabilities, we note that the estimates of the static MM estimator  (Figure \ref{fig:static}) are slightly inferior to that of the dynamic variant (Figure \ref{fig:dynamic}). As aforementioned, this is because the dynamic MM estimator allows for mode transitions through a Markovian jump process where the transition matrix can be used as a design tool or to incorporate prior knowledge about the mode switching process. In this example, the transition matrix is chosen as
$P_T=\begin{bmatrix}  0.7 & 0.15 & 0.15\\ 0.399 & 0.6& 0.001\\ 0.399& 0.001 & 0.6\end{bmatrix}$.

\section{Conclusion} \label{sec:conclusion}
This paper presented a multiple-model estimation algorithm for simultaneously estimating the mode, input and state of hidden mode switched linear stochastic systems with unknown inputs. 
We defined the notion of a generalized innovation sequence, which we then show to be a Gaussian white noise. Next, we exploited the whiteness property of the generalized innovation to form likelihood functions for determining mode probabilities. 
Finally, we investigated the asymptotic behavior, i.e., the mode distinguishability property, of the proposed algorithm.
Simulation results for vehicles at an intersection with switching driver intentions demonstrated the effectiveness of the proposed algorithm. 

%

\section*{Acknowledgments}
This work was supported by NSF grant CNS-1239182. M. Zhu is partially supported by ARO W911NF-13-1-0421 (MURI) and NSF grant CNS-1505664.

\bibliographystyle{unsrt}
\small
\bibliography{biblio}
\end{document}